\documentclass{article}
\usepackage{amsmath,amssymb,amsthm,graphicx,a4wide}
\usepackage[small,bf]{caption}
\usepackage{enumerate}
\usepackage{subfig}
\usepackage{titlesec}

\usepackage{array}
\theoremstyle{plain}
\newtheorem{thm}{Theorem}
\newtheorem{lemma}[thm]{Lemma}

\newtheorem{cor}[thm]{Corollary}

\newtheorem{exe}{Exercise}

\theoremstyle{definition}

\theoremstyle{remark}

\newcommand{\half}{{1/2}}
\newcommand{\third}{{1/3}}
\newcommand{\tthirds}{{2/3}}
\newcommand{\fsixths}{{5/6}}

\newcommand{\prn}[1]{\left(#1\right)}

\newcommand{\abs}[1]{\left|#1\right|}
\newcommand{\brk}[1]{\left[#1\right]}

\newcommand{\sups}[1]{{${}^{\text{#1}}$}}

\newcommand{\goto}{{\rightarrow}}

\newcommand{\Ai}{{\mathop{\mathrm{Ai}}}}
\newcommand{\Bi}{{\mathop{\mathrm{Bi}}}}
\newcommand{\tJ}{{\tilde J}}
\newcommand{\tY}{{\tilde Y}}

\newcommand{\UU}{{\mathcal{U}}}

\newcommand{\WW}{{W}}
\newcommand{\cW}{{\cal W}}

\newcolumntype{C}{>{\(}c<{\)}} 
\newcolumntype{L}{>{\(}l<{\)}} 
\newcolumntype{R}{>{\(}r<{\)}}

\newcommand{\eps}{{\varepsilon}}
\newcommand{\Oh}{{\mathcal{O}}}
\newcommand{\junk}[1]{}

\newcounter{mysubsubsec}[subsection]
\renewcommand{\themysubsubsec}{{\itshape \alph{mysubsubsec}}}
\newcommand{\mysubsubsec}[1]{%
\paragraph%
\noindent%
\stepcounter{mysubsubsec}%
\noindent{\itshape(\themysubsubsec) #1: }}

\title{The hanging thin rod: \\ A singularly perturbed eigenvalue problem}
\author{Yossi Farjoun\footnotemark[2] \and David
  G. Schaeffer\footnotemark[3]}
\begin{document}
\maketitle
\begin{abstract}
We study the vibrations of a hanging thin flexible rod, in which the dominant restoring force in most of the domain is tension due to the weight of the rod, while bending elasticity plays a small but non-negligible role.
We consider a linearized description, which we may reduce to an eigenvalue problem. 
We solve the resulting singularly perturbed problem asymptotically up to the first modification of the eigenvalue.
On the way, we illustrate several important problem-solving techniques: modeling,
nondimensionalization, scaling, 
and especially use of asymptotic series.
\end{abstract}
\footnotetext[2]{G.\ Mill\'an Institute of Fluid Dynamics, Nanoscience, and Industrial Mathematics \\
Universidad Carlos III de Madrid,
Avenida de la\ Universidad 30, Legan\'es, Spain, 28911 \\
Corresponding Author: \texttt{yfarjoun@ing.uc3m.es}}
\footnotetext[3]{Duke University, Department of Mathematics and Center for Nonlinear and
Complex Systems,\\ Durham NC 27708-0320, USA}

\section{Introduction}
Linear transverse vibrations of a uniformly stretched string are modeled with the wave equation
\begin{equation} \label{eq:wave}
	\rho w_{tt} - T w_{yy} = 0,
\end{equation}
where $w(y,t)$ is the displacement of the string, $\rho$ the mass per unit length, and $T$ the tension.
In \eqref{eq:wave} the string is assumed perfectly flexible; if resistance to bending is not negligible (e.g., for a rod), then the equation requires an additional, fourth-order, term, 
\begin{equation} \label{eq:beambending}
 	\rho w_{tt} - T w_{yy} + E I w_{yyyy} = 0,
\end{equation}
where $E$ is Young's modulus and $I$ is the cross-sectional moment of area.
If the tension depends on position $y$, then the middle term of \eqref{eq:beambending} is replaced by a variable-coefficient 
operator in divergence form;
in particular, if the rod is vertical so that the tension at a point $y$ results solely from the weight of the rod below it, located say in $[0,y]$, then the tension is $T(y)= \rho g y$ and \eqref{eq:beambending} is modified to 
\begin{equation} \label{eq:wave:flag:unscaled}
 	\rho w_{tt} - \rho g (y w_y)_y + E I w_{yyyy} = 0,
\end{equation}
where $g$ is the acceleration of gravity.
In this paper we study \eqref{eq:wave:flag:unscaled} in the asymptotic limit where the bending forces are small compared to tension, using singular perturbations%
\footnote{We consider only deflections in one transverse direction.  In the linear approximation assumed in \eqref{eq:wave:flag:unscaled}, vibrations in the two transverse directions are independent.}.

Our interest in this problem originated from Manela and Howe's paper
\cite{MH2009}  simulating the
waving of a flag in the wind. 
The displacement of the flag was represented
as an expansion in terms of eigenfunctions of the second-order linear operator
\begin{equation*}
  w \mapsto \frac{\partial}{\partial y} \prn{y \frac{\partial w}{\partial y}  };
\end{equation*}
i.e., the part of (\ref{eq:wave:flag:unscaled}) due to tension alone, neglecting bending resistance.
It was found that this expansion converged very poorly, if at all.
In this paper we study the effect of a small bending resistance on the eigenfunctions, without
reference to wind-driven motion.

The problem of fluttering
plates and rods driven by a fluid has been studied earlier in various
places. In a recent paper, Argentina and Mahadevan \cite{AM2005} study the problem
where the bending rigidity dominates, and in a previous paper, Dowling \cite{Dowling1987}
uses matched asymptotic expansions to solve a slightly different
problem where the fluid-loading is taken into account. This has the
consequence of moving the main singular point into the bulk which
results in significantly different behavior.

\section{Preliminary Analysis} 

\subsection{Boundary Conditions}
Equation \eqref{eq:wave:flag:unscaled} must be supplemented with two boundary conditions at each end of the domain, say $0 \le y \le L$.
At the free end $y=0$  there should be no bending moment (i.e., $E I w_{yy}(0,t)=0$) and no force.
At a point where the tension vanishes, the force is given by the negative derivative of the bending
moment, so we obtain the second boundary condition $w_{yyy}(0,t)=0$. 

At $y=L$ we consider two distinct possibilities for boundary conditions: clamped and pinned. 
Both imply that there is no deflection at $y=L$; thus, $w(L,t)=0$.
For the second boundary condition, clamped imply zero slope (i.e., $w_y(L,t)=0$), and pinned imply zero
bending moment  (i.e., $w_{yy}(L,t)=0$). 

Thus, in summary, the two possible sets of boundary conditions are 
\begin{equation}
  w_{yy}(0,t)=0, \quad w_{yyy}(0,t)=0,\quad  w(L,t)=0, \; \mbox{plus} \left\{\begin{alignedat}{3}
  \mbox{either}&	& \quad w_y(L,t)&=0,& \quad&\text{(Clamped)}\\
  \mbox{or}&	& \quad w_{yy}(L,t)&=0. &\quad &\text{(Pinned)}
\end{alignedat}\right.
\label{eq:BC:Summary}
\end{equation}

\subsection{Non-dimensionalization and Scaling}
Despite there being several parameters that control the
behavior of \eqref{eq:wave:flag:unscaled}, these may be reduced to a single non-dimensional ``group.''
Let us scale $y$ by the length $L$ and $t$ by $L/c$, where $c=\sqrt{gL}$ specifies the order of magnitude of the speed tension waves; i.e., let
\begin{equation} \label{eq:scaling}
	\tilde{y} = \frac{y}{L}, \quad \tilde{t} = \sqrt{\frac{g}{L}} t.
\end{equation}
The reason we use the timescale of the tension waves is that we are
interested in looking at the \emph{thin rod} case, where the
elasticity plays a small role in the dynamics. 
(We do not scale $w$ since in a linear equation this would not change anything.)

Substituting \eqref{eq:scaling} into \eqref{eq:wave:flag:unscaled} yields
\begin{equation}
\label{eq:wave:flag}
   w_{\tilde{t} \tilde{t}}-(\tilde{y} w_{\tilde{y}})_{\tilde{y}}
		+\eps w_{\tilde{y} \tilde{y} \tilde{y} \tilde y}=0,\qquad \mbox{where } \eps=\frac{EI}{\rho g L^3}.
\end{equation}
The dimensionless constant $\eps$ 
compares the importance of bending elasticity ($\Oh(EI/L^2$)) with the maximum tension
($\rho g L$).
Of course $\eps$ is small if $L$ is large.
Let us also examine the dependence of $\eps$ on $a$, the
width of the rod.
The second moment of area is given by
\begin{equation}
  \label{eq:second:moment:area}
  I = \int_C x^2 dx dz,
\end{equation}
where $C$ is the cross-sectional area of the rod and $x$ is the direction of bending.
This moment scales like $a^4$, while $\rho$, mass \emph{per unit length}, scales like $a^2$. 
Thus, $\eps$ contains an implicit factor of $a^2$%
\footnote{In (\ref{eq:wave:flag:unscaled}), the coefficient $EI$ of the fourth derivative is appropriate for a solid rod but not for 
a \emph{cable} or \emph{string}---i.e., many small fibers twisted together.  The bending resistance of
any one fiber in a string is all but infinitesimal; the primary resistance comes from the friction of fibers sliding over one another.  
While it is 
difficult to calculate the bending resistance of such a collection of fibers, this resistance is \emph{much} smaller than the Young's modulus 
times the 
area moment of the whole cable.  Thus, appropriate values of $\eps$ for a string may be very small indeed.  For accurate modeling of a 
string it might be necessary to include a friction term, say proportional to $w_t$, in (\ref{eq:wave:flag:unscaled}).  Such a term would not 
change the eigenfunctions found below, and its effect on the time dependence (\ref{eq:separation:time}) is easily calculated.}.

Below we omit the tildes from \eqref{eq:wave:flag}.

\subsection{Separation of Variables}

We look for a solution of \eqref{eq:wave:flag} in separated form $ w(y,t)=u(y) \times \Omega(t)$, and find that $u$ must satisfy
\begin{equation}
  \label{eq:singularly_perturbed}
  \eps u'''' - (yu')' = \lambda u, \qquad 0<y<1,
\end{equation}
where $\lambda$ is the eigenvalue parameter, and 
\begin{equation}
  \label{eq:separation:time}
  \Omega(t)=\Omega_0 e^{\pm i\sqrt{\lambda}t}.
\end{equation}
Note that \eqref{eq:singularly_perturbed} is singular for two reasons: (i)~$\eps$ multiplies the highest-order
derivative in the equation and (ii)~the coefficient of leading-order derivative in the reduced equation (after setting $\eps=0$), 
\begin{equation}
  \label{eq:reduced}
   - (yu')' = \lambda u,
\end{equation}
vanishes at one end of the interval.
In the remainder of the paper we solve asymptotically this \emph{singularly perturbed} eigenvalue problem, subject to the boundary conditions
\begin{equation}
  \label{eq:boundary:conditions}
  u''(0)=0, \quad u'''(0)=0,\quad  u(1)=0, \; \mbox{plus } \left\{\begin{alignedat}{3}
  \mbox{either}&	&\quad u'(1)&=0,& \quad&\text{(Clamped)}\\
  \mbox{or}&	&\quad u''(1)&=0.& \quad&\text{(Pinned)}
\end{alignedat}\right.
\end{equation}
Since eigenfunctions are determined only up to a multiplicative
constant, we add the normalization
\begin{equation}
 \label{eq:extraBC}
	u(0) = 1,
\end{equation}
thereby selecting a unique solution of this problem.

Incidentally, we claim that \eqref{eq:singularly_perturbed} with either boundary conditions \eqref{eq:boundary:conditions} is self-adjoint and in fact positive-definite.
This may be proved with the usual integration-by-parts argument, with
one subtlety: usually, boundary terms vanish because of the boundary conditions, but the term $(yu')'$ makes no contribution at $y=0$ only because the coefficient $y$ vanishes there.
Thus, all eigenvalues of this problem will be positive real.

\bigskip
The remainder of this paper is organized as follows.
In Section~\ref{sec:first_attempts} 
we propose a naive derivation of the limiting behavior of the eigenvalues and eigenfunctions 
of \eqref{eq:singularly_perturbed}  as $\eps \to 0$, 
and we present some numerical results.
In Section~\ref{sec:preparation} we find the boundary-layer scalings for the asymptotic analysis and solve the reduced equations in various regimes.
In Sections~\ref{sec:clamped} and~\ref{sec:pinned} we derive matched asymptotic series for the eigenvalues and eigenfunctions
with clamped and pinned boundary conditions, respectively.
Finally, in Section~\ref{sec:closing} we discuss the applicability of small-$\eps$ asymptotics for later eigenvalues in the sequence of eigenvalues.

In addition to its research interest, this problem provides a relatively simple example of a
matched asymptotic expansion that requires logarithmic terms.
Thus for pedagogical reasons, we strive for careful, thorough explanations, and we have included
several exercises for the readers to sharpen their understanding.

\section{First Attempts}
\label{sec:first_attempts}
\subsection{The Naive Solution: \protect$\eps=0$}
\label{subsec:naive}
Setting $\eps$ equal to zero in \eqref{eq:singularly_perturbed} yields equation \eqref{eq:reduced}.
Of course no solution of this equation can satisfy all the boundary conditions \eqref{eq:boundary:conditions}.
It is natural to conjecture that at $y=1$ the ``most physical'' solution will at least satisfy the lower-order boundary condition there
 \begin{equation}
 \label{eq:red:BC_1}
	u(1) = 0.
\end{equation}
Because \eqref{eq:reduced} is singular at $y=0$, it is unclear what boundary conditions, if any, ought to be imposed there.

In fact, \eqref{eq:reduced} can be solved explicitly in terms of Bessel functions of order zero.
The connection to Bessel functions may be motivated by expanding all terms in \eqref{eq:reduced},
\begin{equation} \label{eq:reduced:expanded}
 	y u'' + u' + \lambda u = 0
\end{equation}
and observing that this equation, like Bessel's equation, has exactly two singular points, a regular one at $y=0$ and an irregular one at infinity.
Moreover, the indicial equation of \eqref{eq:reduced:expanded}, obtained by seeking a series solution
\begin{equation}
 		u(y) = y^p \; \sum_{j=0}^\infty c_j y^j
\end{equation}
where $c_0 \ne 0$, has a double root $p=0$, again like Bessel's equation of order zero. 
To make the reduction, consider the substitution $u(y)=v(C y^p)$; a simple calculation shows that if
 one chooses $p=\tfrac12, C=2\sqrt{\lambda}$, then $v$ satisfies Bessel's equation $v''+x^{-1}v'
+v=0$. 
Thus, the general solution of \eqref{eq:reduced} is
\begin{equation}
 \label{eq:lincombBessel}
 	u(y) = a J_0\prn{2\sqrt{\lambda y}} + b Y_0\prn{2\sqrt{\lambda y}}
\end{equation}
for arbitrary constants $a$ and $b$.
As $y \to 0$, $J_0$ is smooth, but $Y_0$ blows up logarithmically (see Appendix~\ref{subsec:bessel}).
Unless $a=b=0$, neither 
of the boundary conditions \eqref{eq:boundary:conditions} at $y=0$ can be satisfied.
Let us require that $b=0$ so that at least $u$ remains bounded (and \eqref{eq:extraBC} is meaningful).
If the boundary condition \eqref{eq:red:BC_1} at $y=1$ is to be satisfied, then $\lambda$ must satisfy
\begin{equation}
 \label{eq:zeroBessel}
	J_0\prn{2\sqrt{\lambda}}=0.
\end{equation} 

This intuitive discussion gives eigenvalues $\lambda\approx  1.4458,
7.6178, 18.721, \ldots $
with eigenfunctions $J_0(2\sqrt{\lambda y})$,
for \emph{either clamped or pinned} boundary conditions. 
The eigenfunctions are graphed in Figure~\ref{fig:besselJ_0}.
Despite the many loose threads in the argument, exactly these
eigenvalues and eigenfunctions will emerge as the leading-order term
in our asymptotic solution,
thus providing a far more satisfactory derivation. 
Moreover, the difference between clamped and pinned BC will emerge in the higher-order terms of the series.
However, before tackling the asymptotics, we turn to numerics.
\begin{figure}[th]
  \centering
  \scalebox{.9}{\includegraphics{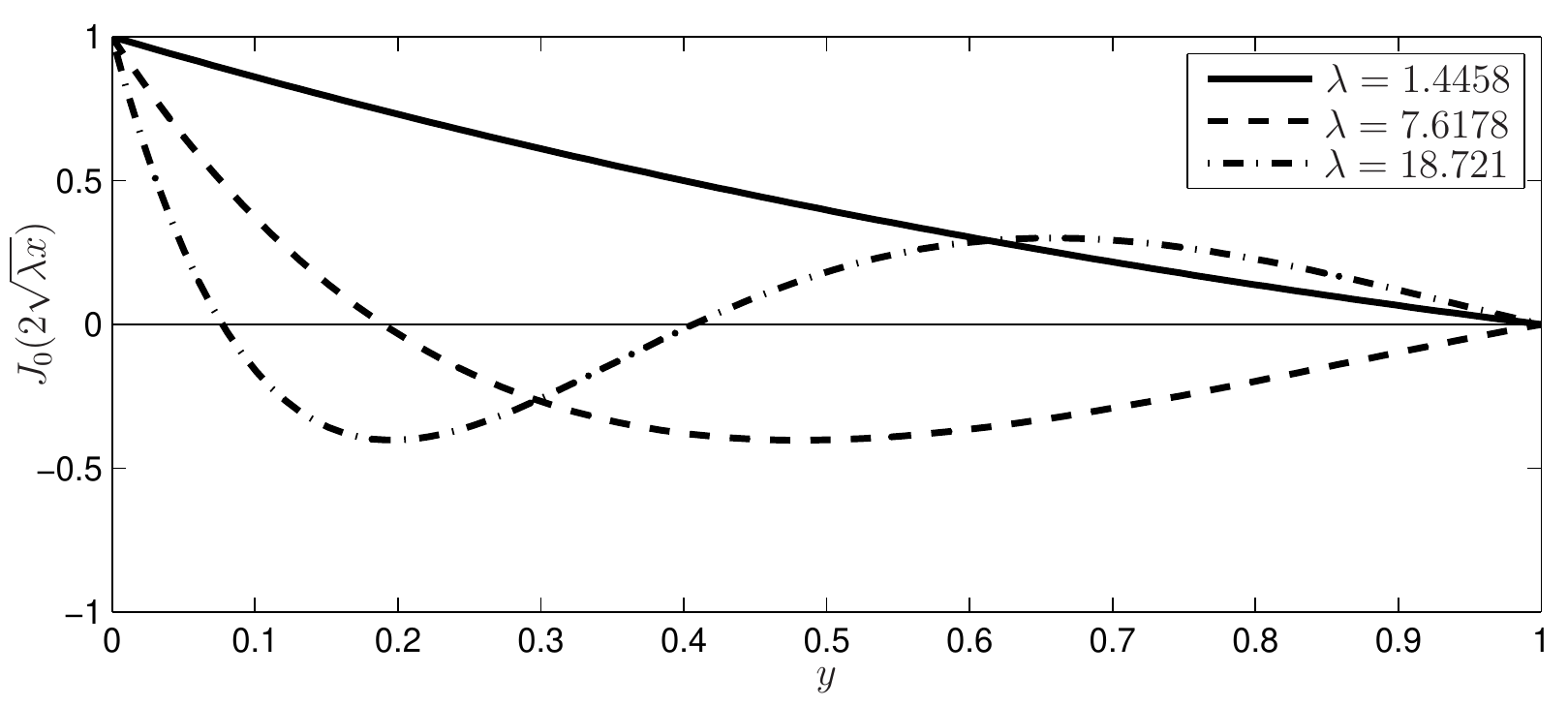}}
  \caption{The naive solution of the reduced equation
    (\ref{eq:reduced}) using the three smallest roots of the Bessel
    function $J_0(2\sqrt{\lambda})$.}
  \label{fig:besselJ_0}
\end{figure}

\subsection{Numerical Solutions}

In an exploratory numerical code, we approximated (\ref{eq:singularly_perturbed}) by differences on a uniform grid.
However, especially for small $\eps$, we found this simple approach gave unreliable results, even when up to 50,000 grid points were used.
For the data presented below, we used the boundary-value solver \texttt{bvp5c} in Matlab. 
The interval $0 \le y \le1$ was divided into 3 distinct subintervals
(corresponding to the 2 boundary layers and the bulk domain,
introduced below), which were connected with internal boundary conditions.
Results from the uniform-grid code were used as initial guesses for
the internal iterative solver.
Very small values of $\eps$ were approached by continuation.

Graphs of the computed eigenfunctions resembled the naive eigenfunctions in Figure~\ref{fig:besselJ_0}, but graphs of their 
derivatives differed substantially.
This issue is explored in some detail in
Subsections~\ref{subsec:composite:approximation} and~\ref{subsec:pinned:2/3}(e).

Figure~\ref{fig:lambda:bvp} shows the divergence of computed
eigenvalues from the naive approximation; specifically, a log-log plot of
$\abs{\lambda^{(n)}(\eps) - \lambda^{(n)}(0)}/\lambda^{(n)}(0)$ vs. $\eps$, where
$\lambda^{(n)}(0)$ is the $n$\sups{th} root  of $J_0(2\sqrt{\lambda})$.
The results support the naive analysis and also suggest that
\begin{equation}
  \label{eq:eigen_values:general}
  \lambda^{(n)}(\eps) = \lambda^{(n)}(0)  +  \Oh(\eps^p)
\end{equation}
where $p=1/2$ or $p=1$ for clamped or pinned boundary conditions, respectively.
This suggestion is confirmed by the asymptotic series solution constructed below.

\begin{figure}[th]
  \subfloat[Clamped]{\label{fig:lambda:clamped:bvp}
  \scalebox{.8}{\includegraphics{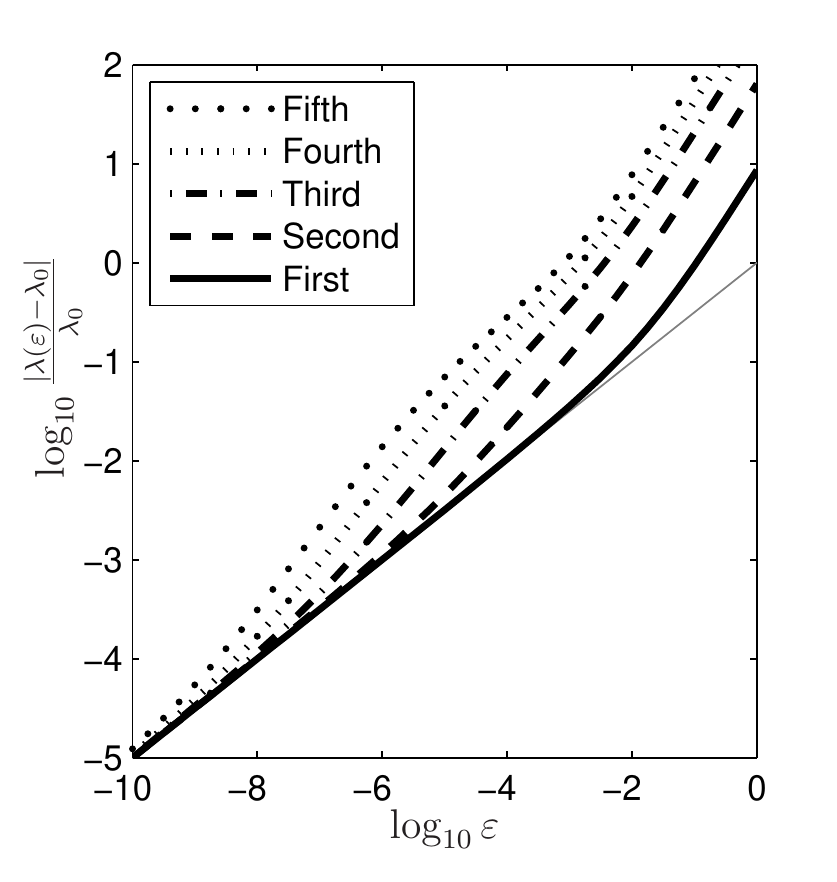}}}
  \subfloat[Pinned]{\label{fig:lambda:pinned:bvp}
  \scalebox{.8}{\includegraphics{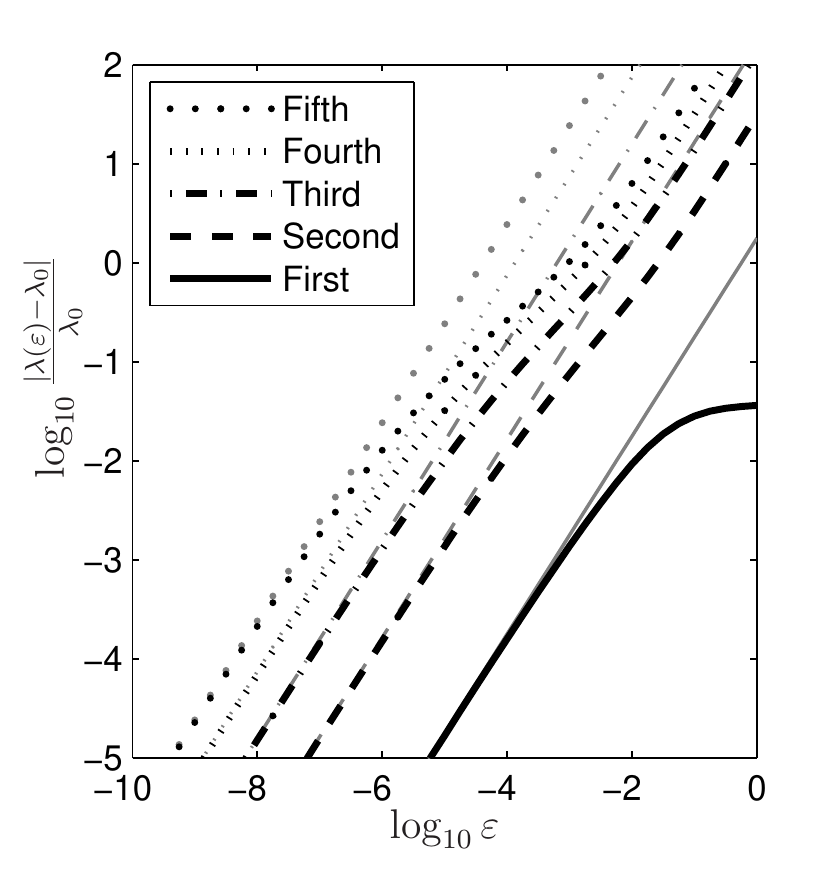}}}
  \caption{A log-log plot of the relative error
    $e(\eps) = \abs{\frac{\lambda^{(n)}(\eps)-\lambda^{(n)}(0)}{\lambda^{(n)}(0)}}$ for the first five
    eigenvalues.   According to our asymptotic analysis, in the clamped case, 
$e(\eps)\sim(\lambda_\half/\lambda_0) \eps^\half$ where $\lambda_\half$ is given in Table~\ref{tab:asymptotic:series:clamped}; since $\lambda_\half=\lambda_0$, all lines
collapse onto $e(\eps)=\eps^\half$ (which is drawn for reference).  In the pinned case $e(\eps)\sim(\lambda_1/\lambda_0) \eps$ where $\lambda_1$ is given by 
\eqref{eq:pinned:lambda_1}.  Reference lines $e(\eps)=(\lambda_1/\lambda_0) \eps$ are drawn in the figure.}
\label{fig:lambda:bvp}
\end{figure}

\section{Preparation for the Asymptotic Solutions}
\label{sec:preparation}
\subsection{Boundary-Layer Scalings}

In the bulk of the domain, far away from the boundaries, we assume
that $\eps$ times the fourth derivative is small enough that we can ignore this term to
lowest order, obtaining \eqref{eq:reduced}.
As we have seen, solutions of this equation cannot satisfy all the boundary conditions.
We expect therefore that there are two adjustment zones---boundary
layers---connecting the solution in the bulk to the boundaries, in a way that satisfies the boundary conditions.
Near the boundaries, we expect that the fourth derivative becomes large enough 
that, even when multiplied by $\eps$, it cannot be ignored.
As a first step we have to find the boundary-layer scalings that are appropriate near the boundaries. 

To find the
scaling near the $y=0$ boundary, we assume that $y$ is of order
$\eps^p$; specifically that $u(y;\eps)$ can be approximated by $U(X;\eps)$ where
$X=\eps^p y$. To find $p$ we differentiate according to
Eq.~\eqref{eq:singularly_perturbed}  and find that 
\begin{equation}
\underbrace{\eps^{1+4p}U''''\vphantom{)}}_a-\underbrace{\eps^p(XU')'}_b=
	\underbrace{\lambda U\vphantom{)}}_c
\label{eq:boundary:layer}
\end{equation}
where prime indicates differentiation with respect to $X$.
An appropriate scaling must balance Eq.~\eqref{eq:boundary:layer}; i.e., two terms must have the same order in $\eps$ and the third term must have order this high or higher.
Thus, we look at three possible
cases:
\begin{itemize}
\item[$a\approx b$] Thus $1+4p=p$ and so $p=-1/3$; in this case terms
  (a) and (b) are much larger (i.e., lower order) than (c).  Hence it is a good scaling of the
boundary layer. 
\item[$b\approx c$] Thus $p=0$ ({\it i.e.}, no scaling
  happens.) Here the two terms (b) and
  (c) are larger than (a), so this scaling is valid; indeed this is the bulk scaling. 
\item[$a\approx c$] Thus $1+4p=0$ so $p=-1/4$ and so term (b) is much
  larger than (a) and (c), so this scaling does not
  balance.\footnote{This scaling would be valid in the opposite
    asymptotic limit where one considers $\eps$ to be large.}
\end{itemize}
The only boundary layer we find near $y=0$ scales the inner variable $X$
like $\eps^{1/3}$, and the resulting equation is therefore
\begin{equation}
  \label{eq:perturbed:boundary:0}
    U'''' - (XU')' = \eps^\third\lambda U.
\end{equation}

To find the
scaling near at the $y=1$ boundary, we
write the solution $u(y;\eps)$ as $V(Z;\eps)$ where
$Z=\eps^q (1-y)$. From
Eq.~\eqref{eq:singularly_perturbed}  we obtain
\begin{equation}
	\eps^{1+4q}V'''' -\eps^{2q}\prn{\prn{1-\eps^{-q}Z}V'}'= \lambda V.
\label{eq:boundary:layer:1}
\end{equation}
Here there are also three possible cases.
\begin{exe}
In the same way as before, show that the only valid, interesting
scaling for small $\eps$ at the $y=1$ boundary has $q=-1/2$, and that the resulting equation for $V(Z;\eps)$ is
\begin{equation}
  \label{eq:perturbed:boundary:1}
    V'''' - \prn{\prn{1-\eps^\half Z}V'}' = \eps\lambda V.
\end{equation}
\end{exe}

\subsection{Reduced Equations}

In the asymptotic series for the solution, terms of order 1/3 and 1/2 are forced by the boundary layers, and subsequent terms include all sums of integer multiples of 1/3 and 1/2.
Thus, we look for a solution that has the following form: 
\begin{itemize}
\item 
In the bulk we expand the solution $u(y;\eps)$ as
\begin{equation}
u(y;\eps)=u_0(y)+\eps^\third u_\third(y)+\eps^\half
u_\half(y)+\eps^{\tthirds} u_{\tthirds}(y)+\eps^{5/6} u_{5/6}(y)+ \eps u_1(y) +\ldots
\label{eq:asymptotic:series:u}
\end{equation}
\item 
Near the $y=0$ boundary, the solution is  approximated by a similar
asymptotic series for $U(X;\eps)$, with $X=\eps^{-1/3} y$.
\item 
Similarly, for the solution near $y=1$ with $V(Z;\eps)$, where $Z=\eps^{-1/2} (1-y)$.
\item 
The  eigenvalue%
\footnote{Of course, there is a sequence of eigenvalues
  with corresponding eigenfunctions. We drop the superscript in the
  $\lambda^{(n)}$ notation for sake of a cleaner presentation.}%
 $\lambda$ is itself expanded as an asymptotic series: 
\begin{equation}
\lambda(\eps)=\lambda_0+\eps^\third \lambda_\third+\eps^\half \lambda_\half+\eps^\tthirds \lambda_{\tthirds}+\eps^\fsixths \lambda_{\fsixths}+\eps \lambda_1+\ldots
\label{eq:asymptotic:series:lambda}
s\end{equation}
\end{itemize}
Below, we will see that the bulk series needs to be augmented by a term of order $\eps \log \eps$,
and logarithmic terms will be needed at higher order as well.
However, for now we wait for this complication to arise naturally.

For the three domains, we substitute the asymptotic expansion of
$u$, $U$, or $V$, and of $\lambda$ into the appropriate differential
equation and collect powers of $\eps$. 
Doing this for an expansion up to $\eps^1$ results in 
\begin{align}
  \label{eq:reduced:equations:u}
  &\begin{array}{rl}
\eps^0:& (y u_0')'+\lambda_0 u_0=0\\
\eps^\third:&  (y u_\third')' + \lambda_0 u_\third= -\lambda_\third u_0 \\
\eps^\half:& (y u_\half')' + \lambda_0 u_\half  = -\lambda_\half u_0  \\
\eps^\tthirds:& (y u_\tthirds')' + \lambda_0 u_\tthirds = -\lambda_\third u_\third - \lambda_\tthirds u_0\\
\eps^{5/6}:& (y u_\fsixths')' + \lambda_0 u_\fsixths = -\lambda_\third
u_\half-\lambda_\half u_\third -\lambda_\fsixths u_0 \\
\eps^1:& (y u_1')'+ \lambda_0 u_1=  u_0'''' -\lambda_\third u_\tthirds -
\lambda_\half u_\half- \lambda_\tthirds u_\third -\lambda_1 u_0
\end{array}
\intertext{for $u$,}
  \label{eq:reduced:equations:U}
  &\begin{array}{rl}
\eps^0:& U_0'''' - (X U_0')'=0\\
\eps^\third:&  U_\third'''' - (X U_\third')'=\lambda_0 U_0\\
\eps^\half:&  U_\half'''' - (X U_\half')'=0\\
\eps^\tthirds:& U_\tthirds'''' - (X U_\tthirds')'=\lambda_0 U_\third+\lambda_\third U_0\\
\eps^{5/6}:& U_\fsixths'''' - (X U_\fsixths')'=\lambda_0
U_\half+\lambda_\half U_0 \\
\eps^1:&  U_1'''' -(X U_1')'=  \lambda_0 U_\tthirds +\lambda_\third U_\third +
\lambda_\tthirds U_0
\end{array}
\intertext{for $U$, and}
\label{eq:reduced:equations:V}
&\begin{array}{rl}
\eps^0:& V_0'''' - V_0''=0\\
\eps^\third:&  V_\third'''' - V_\third''=0\\  
\eps^\half:&  V_\half'''' -V_\half''= - (Z V_0')'\\
\eps^\tthirds:& V_\tthirds'''' - V_\tthirds''=0\\
\eps^{5/6}:&V_\fsixths'''' - V_\fsixths''=-(Z V_\third')'  \\
\eps^1:&  V_1'''' - V_1''=  \lambda_0 V_0 -(Z V_\half')'
\end{array}
\end{align}
for $V$.

\subsection{Solutions of the Reduced Bulk Equations}
\label{subsec:sol:bulk}

In the bulk, the reduced equations (\ref{eq:reduced:equations:u}) have the form 
$(y u')'+\lambda_0 u=g$, or expanding the derivative and dividing by $y$,
\begin{equation} \label{eq:inhomog-bulk}
 	u'' +\frac{1}{y} u' + \frac{\lambda_0}{y} u = \frac{g(y)}{y}.
\end{equation}
We saw in Subsection~\ref{subsec:naive} that in the homogeneous case, $g \equiv 0$, the general solution of (\ref{eq:inhomog-bulk}) is given by (\ref{eq:lincombBessel}).
To shorten the notation and form a more convenient basis for the solution space, we define
\begin{align}
\label{eq:Bessel:short:J}
  \tJ(y) &= J_0(2\sqrt{\lambda_0 y}) \\ 
\label{eq:Bessel:short:Y}
  \tY(y) &=\pi Y_0(2\sqrt{\lambda_0 y}) -(\log \lambda_0+2\gamma)\tJ(y)
\end{align}
where $\gamma$ is Euler's constant.
The modification of $Y_0$ simplifies the asymptotic behavior as $y \to 0$ (see Appendix~\ref{subsec:bessel}):
\begin{equation} \label{eq:Y-asymp}
 		\tY(y)=\log(y)+\Oh(y \log y).
\end{equation}
Note that this notation hides $\lambda_0$, an unknown constant that still needs to be determined.

In the inhomogeneous case we have
\begin{lemma}
\label{lem:w}
Given a continuous
function $g(y)$, there exists exactly one solution of \eqref{eq:inhomog-bulk}, say $w(y)$, that is continuous on the closed interval $[0,1]$ and satisfies $w(0)=0$.
\end{lemma}
For this singular problem, just one boundary condition suffices to determine the solution uniquely.

\begin{proof}[Proof of Lemma~\ref{lem:w}:] {\itshape (Existence) }
Using variation of coefficients, we find that one solution of \eqref{eq:inhomog-bulk} is given by 
\begin{align}
  \label{eq:w:for:j_0}
   w(y)&=c_1(y) \tJ(y) + c_2(y) \tY(y),
\intertext{where}
\label{eq:w:for:j_0:c1}
c_1(y)&=-\int_0^y\frac{1}{x \cW(x)}\tY(x)g(x)dx=-\int_0^y
\tY(x)g(x) dx\\
\label{eq:w:for:j_0:c2}
c_2(y)&=\int_0^y \frac{1}{x \cW(x)}\tJ(x)g(x)dx=\int_0^y
\tJ(x)g(x)dx.
\end{align}
Here $\cW=\tJ \tY'-\tY \tJ'$ denotes the Wronskian of $\tJ(y)$ and $\tY(y)$, and we have used the fact $\cW(y)=1/ y$ (see Appendix~\ref{subsec:bessel}).
Regarding the first term in \eqref{eq:w:for:j_0}, although the
integrand defining $c_1$ in \eqref{eq:w:for:j_0:c1} diverges logarithmically at $x=0$, the integral converges for all $y \in [0,1]$ and vanishes if $y=0$.
Regarding the second term \eqref{eq:w:for:j_0:c2}, although $\tY(y)$ blows up logarithmically as $y \to 0$, the coefficient satisfies $c_2(y) = \Oh(y)$ so the product 
$c_2 \tY$ vanishes at $y=0$.
Thus,  \eqref{eq:w:for:j_0} solves \eqref{eq:inhomog-bulk} and
satisfies the boundary conditions in Lemma~\ref{lem:w}.

{\it (Uniqueness)} The general solution of \eqref{eq:inhomog-bulk} has the form 
\begin{equation*}
 	u(y) = w(y) + a \tJ(y) + b \tY(y)
\end{equation*}
for arbitrary constants $a$ and $b$.
For this function to be continuous we need $b=0$, and for it to vanish at the origin, we need $a=0$.
\end{proof}

\begin{cor} \label{cor:nonzero}
If $\lambda_0$ satisfies $J_0(2 \sqrt{\lambda_0})=0$
and $g(y)=\tJ(y)$, then the solution
to Eq.~\eqref{eq:inhomog-bulk},  $w(y)$, does not vanish at $y=1$.
\end{cor}
\begin{proof}
Since $\tJ(1)=0$, we have from \eqref{eq:w:for:j_0} that 
\begin{equation}
  \label{eq:1}
  w(1)= c_2(1) \tY(1)    =\tY(1)\int_0^1\tJ^2(x)\, dx\ne0. 
\end{equation}
Of course the integral is positive, and $\tY$ cannot vanish at $y=1$
since $\tJ$ already vanishes there and the pair's Wronskian is non-zero.
\end{proof}

\subsection{Solution of the Reduced Boundary Layer Equations} 
\label{ssec:BL:Equations}

Equations~\eqref{eq:reduced:equations:V} for the boundary layer at $y=1$, which have constant coefficients, do not require any special discussion.
Therefore, we focus here on the boundary layer at $y=0$.

The homogeneous versions of equations \eqref{eq:reduced:equations:U} have the form
\begin{equation} \label{eq:reduced:equations:U1}
	U'''' - (XU')' = 0. 
\end{equation}
We need to find four linearly independent solutions of this equation.
By inspection, $U^{(1)} \equiv 1$ is one such solution.
Observe that, if we define $\WW=U'$, we may rewrite \eqref{eq:reduced:equations:U1} as
\begin{equation} \label{eq:reduced:equations:U2}
 	\frac{d}{dX} \brk{\WW'' - X\WW }= 0.
\end{equation}
Now $\WW''-X\WW=0$ is Airy's differential equation, whose solution
space is spanned by the Airy functions, $\Ai$ and $\Bi$.
(See Appendix~\ref{subsec:airy} for the definition and some elementary properties of these special functions.)
Thus, 
\begin{equation} \label{eq:U2U3}
 U^{(2)}(X) = \int_0^X \Ai(x)dx, \qquad U^{(3)}(X) = \int_0^X \Bi(x)dx
\end{equation}
provide two more linearly independent solutions.
Since $\Bi$ grows super-exponentially as $X \to \infty$, the solution $U^{(3)}$ cannot be matched to any solution in the bulk.
By contrast, $\Ai$ \emph{decays} super-exponentially so the integral to infinity converges; in fact, by  \eqref{eq:Airy:Integrals}
\begin{equation} \label{Ai-infty}
 \lim_{X \to \infty} U^{(2)}(x) = 1/3.
\end{equation}
Incidentally, for use in the boundary conditions
(\ref{eq:boundary:conditions}, \ref{eq:extraBC}), we claim that
\begin{equation} \label{eq:U2BC}
 U^{(2)}(0) = 0, \qquad
U^{(2)}{}''(0) = \Ai'(0) \neq 0, \qquad
U^{(2)}{}'''(0) = 0.
\end{equation}
The first relation is trivial; the second may be derived by differentiating
\eqref{eq:U2U3} twice; and third may be derived by differentiating
\eqref{eq:U2U3} thrice and invoking Airy's differential equation.

To obtain a fourth independent solution, we satisfy \eqref{eq:reduced:equations:U2} by requiring that $\WW''-X\WW=-1$, or since $\WW=U'$,
\begin{equation} \label{eq:reduced:equations:U4}
 	 	U''' - XU' = -1.
\end{equation}
Specifically we let $U^{(4)}= \Psi$ where $\Psi$ satisfies the following
\begin{lemma} \label{lem:psi}
There is a unique solution $\Psi(X)$ of \eqref{eq:reduced:equations:U4} such that 
\begin{alignat*}{2}
\text{(a) } \; \Psi(0) &= 0,  \quad\text{(b) } \; \Psi''(0)=0  \quad \text{ and } \\
 \qquad \text{(c) }  \;  \Psi(X) &= \log(X) + \Oh(1) \; \text{ as } \; X\goto\infty.&
\end{alignat*}
\end{lemma}

This lemma is proved in Appendix~\ref{app:proof:Psi}, and  
Figure~\ref{fig:Psi} shows the graph of $\Psi$, obtained numerically.
Incidentally, the numerics indicate that as $X \to \infty$
\begin{equation} \label{eq:U4-value}
 \Psi(X) = \log(X) + \Psi_\infty + o(1)  \qquad \text{where} \qquad
 	\Psi_\infty \approx 1.3556,
\end{equation}
while it follows from \eqref{eq:reduced:equations:U4} that
\begin{equation} \label{eq:U4-der}
	\Psi'''(0)=-1.
\end{equation}

Below we also solve inhomogeneous versions of equations~\eqref{eq:reduced:equations:U}.

\section{Asymptotics for the Clamped Case}
\label{sec:clamped}
In this section, we calculate terms in the asymptotic series for the
clamped case through the first non-trivial correction to the eigenvalue: i.e., order $1/2$.
The results are summarized in Table~\ref{tab:asymptotic:series:clamped}.

\begin{table}[h!]
\setlength{\extrarowheight}{3pt} 
\centerline{
\begin{tabular}{|C|*{4}{C}|}
\hline
\alpha& \lambda_\alpha &u_\alpha&U_\alpha&V_\alpha\\
\hline
0&\lambda_0& \tJ(y)& 1& 0\\
\tfrac13& 0&0 &-\lambda_0 X&0\\
\tfrac12&\lambda_0& - \lambda_0 w(y)& 0& \tJ'(1)\brk{1-Z-e^{-Z} }\\[1ex]
\hline
\end{tabular}}
\caption{Coefficients of the asymptotic expansion for the clamped
  case.  The leading-order eigenvalue $\lambda_0$ is a root of the equation $J_0(2 \sqrt{\lambda_0}) =0$. $\; \tJ$ and $w$ are defined in \eqref{eq:Bessel:short:J} and \eqref{eq:w-third}, respectively.}
\label{tab:asymptotic:series:clamped}
\end{table}


\subsection{Zeroth-Order Solution}
\label{subsec:clamped:0}
\mysubsubsec{Boundary-layer near $y=0$}
As discussed in Section~\ref{ssec:BL:Equations}, the solution of the $\eps^0-$equation in \eqref{eq:reduced:equations:U} is a linear combination
\begin{equation} \label{eq:U_0:gen}
	U_0(X) = A +B \int_0^X \Ai(x)dx + C \Psi(X),
\end{equation}
the integral of $\Bi$ having been excluded as unsuitable for matching.
To satisfy the boundary conditions~(\ref{eq:boundary:conditions}, \ref{eq:extraBC}), we require
\begin{equation}  \label{U_0:bc}
	U_0(0)=1, \qquad U_0''(0)=0, \qquad U_0'''(0)=0.
\end{equation}
By substituting (\ref{eq:U_0:gen}) into the boundary conditions and
recalling the derivatives of $\int \Ai$ and $\Psi$ at $X=0$, we deduce
that $A=1, B=C=0$, that is,
\begin{equation}
  \label{eq:U_0}
  U_0(X) \equiv 1.
\end{equation}

\mysubsubsec{Boundary-layer near $y=1$}
The four-dimensional solution space of the $\eps^0$ equation in \eqref{eq:reduced:equations:V} is spanned by $1, Z, e^{-Z}, e^{Z}$.
Excluding $e^Z$ as unsuitable for matching, we write $  V_0(Z)=A + B Z + C e^{-Z}$.
The boundary conditions \eqref{eq:boundary:conditions} for the clamped case,
\begin{equation}
 V_0(0)=0 \qquad V_0'(0)=0,
\end{equation}
allow us to express two of the three arbitrary constants in terms of the third, yielding 
\begin{equation}
  \label{eq:V_0:clamped:form}
  V_0(Z)=A (1 - Z - e^{-Z}).
\end{equation}
The undetermined coefficient $A$ will be found by matching with the bulk solution.

\mysubsubsec{Bulk} 
We already know that the general solution of the $\eps^0-$equation  \eqref{eq:reduced:equations:u} in the bulk is
\begin{equation}
  \label{eq:bulk:0:linear}
  u_0(y)=a \tJ(y) + b \tY(y).
\end{equation}
The two constants will be determined in matching.

\mysubsubsec{Matching}
To match the bulk solution with that of the boundary layers, we compare the ``outer limit'' of the
inner solutions (the boundary-layer solutions) to the ``inner limit''
of the outer solution (the solution in the bulk). 
Thus, near $y=0$ we need, as $\eps \to 0$,
\begin{equation}
  \label{eq:matching:0:0}
  u_0(\eps^{1/3}X)- U_0(X)=o(1) 
\end{equation}  
for an appropriate range of $X$.
Specifically, we require that there exist numbers $p,q$, with $0 \le p < q \le 1/3$, such that (\ref{eq:matching:0:0}) holds for all $X$ 
such that
\begin{equation} \label{rangeX}
 \eps^{-p} \ll X \ll \eps^{-q}.
\end{equation}
In this case we may take $p=0, q=1/3$; i.e., the maximal range.
It follows from (\ref{eq:bulk:0:linear}) that $u_0(y) = a + b \log y + \Oh(y \log y)$ for small $y$.
Hence, since $\eps^{1/3}X \ll 1$,
\begin{equation} 
\label{d1}
 u_0(\eps^{1/3}X) = a + b \log (\eps^{1/3}X) + o(1).
\end{equation}
Unless $b=0$, the logarithm term in (\ref{d1}) is large and moreover depends on $X$.
Therefore for $u_0$ to match onto $U_0(X) \equiv 1$, we need $a=1$ and
$b=0$.

Near $y=1$ we need
\begin{equation}
  \label{eq:matching:0:1}
u_0(1-\eps^{1/2} Z) - V_0(Z)=o(1) \qquad \text{for $Z$ in a range} \qquad \eps^{-p} \ll Z \ll \eps^{-q}
\end{equation}
where  $0 \le p < q \le 1/2$.
We take $p=0, q=1/2$.
Given $a,b$ as above, $u_0(1-\eps^{1/2} Z) = \tJ(1) +o(1)$
while  $ V_0(Z)= A (1 - Z) + o(1)$.
Thus matching the linear term in $V_0$ requires that $A=0$, and then matching the constant terms requires that $\tJ(1)=0$.
Finally, extracting the implicit $\lambda_0$ from the argument of $\tJ$, we obtain the characterization of $\lambda_0$
\begin{equation}
 J_0(2\sqrt{\lambda_0})=0.
\end{equation}
This confirms the \emph{ad hoc} solution we found in Subsection~\ref{subsec:naive}, and it verifies the first row of Table~\ref{tab:asymptotic:series:clamped}.

\subsection{The $\eps^{1/3}$ Correction}
Not much happens at this order, but it serves as practice for later calculations.

\mysubsubsec{Boundary-layer near $y=0$}
At order $\eps^\third$, equation \eqref{eq:reduced:equations:U} is inhomogeneous with the right-hand-side $\lambda_0 U_0 \equiv \lambda_0$, which has the particular solution $U_{p}(X) = - \lambda_0 X$.
$U_\third$ must satisfy boundary conditions analogous to
(\ref{U_0:bc}), except that now $U_\third(0) = 0$ replaces the
condition $U_0(0)=1$. 
Since $U_p$ already satisfies the boundary conditions, we have that
\begin{equation}
  \label{eq:U_third}
  U_\third(X) = - \lambda_0 X.
\end{equation}
As we shall see below, this term merely matches the first derivative
of the bulk solution $u_0$ at $y=0$.

\mysubsubsec{Boundary-layer near $y=1$}
Arguing as in the $\eps^0$ case in Subsection~\ref{subsec:clamped:0}, we deduce that 
\begin{equation}
  \label{eq:V_third:clamped:form}
  V_\third(Z)=A (1 - Z - e^{-Z}).
\end{equation}

\mysubsubsec{Bulk}
The $\eps^{1/3}$ equation \eqref{eq:reduced:equations:u} has the inhomogeneous term $-\lambda_\third u_0$.
Thus, the general solution of this equation is 
\begin{equation}
  \label{eq:bulk:third:linear}
  u_\third(y)= - \lambda_\third w(y) +a \tJ(y) + b \tY(y)
\end{equation}
where we define $w$ as the solution (see Lemma~\ref{lem:w}) of 
\begin{equation} \label{eq:w-third}
(y w')' + \lambda_0 w=  \tJ \qquad \text{such that} \qquad w(0) = 0.
\end{equation}

\mysubsubsec{Matching}
Near $y=0$ we need, as $\eps \to 0$,
\begin{equation}
  \label{eq:matching:third:0}
	(u_0 + \eps^{1/3} u_\third)(\eps^{1/3}X) - (U_0 + \eps^{1/3} U_\third)(X) = o(\eps^{1/3})
\end{equation}  
for $X$ in a range $\eps^{-p} \ll X \ll \eps^{-q}$ where $0 \le p < q \le 1/3$.
We take $p=0, q=1/6$ so that $(\eps^\third X)^2 = o(\eps^\third)$.
Then by Taylor expansion at $y=0$,
\begin{equation} \label{eq:taylor:u:0}
 u_0(\eps^{1/3}X) = u_0(0) + \eps^{1/3} u_0'(0) X + o(\eps^{1/3}).
\end{equation}
Therefore, obtaining $u_0'(0)$ from  \eqref{eq:taylor:J_0} in the
Appendix, and recalling that $w(0)=0$ in (\ref{eq:bulk:third:linear}), we find%
\footnote{The term $\log (\eps^{1/3}X)$ in this equation is a little disturbing.  One expects the function that multiplies $\eps^\third$
to depend on $X$ alone, not $\eps$.  This confusing behavior, a consequence of the singularity of (\ref{eq:reduced}) at $y=0$, is not an issue 
here since the matching implies that $b=0$.  However, exactly this
complication will force us to add a $\eps \log \eps$ term to the series
in Section~\ref{sec:pinned} below.}
\begin{equation}
 (u_0 + \eps^{1/3} u_\third)(\eps^{1/3}X) = 1 + \eps^{1/3}\prn{-\lambda_0 X  +a+b\log (\eps^{1/3}X)} + o(\eps^{1/3}).
\end{equation}
On the other hand, 
\begin{equation}
 (U_0 + \eps^{1/3} U_\third)(X) = 1 -\eps^{1/3}\lambda_0 X.
\end{equation}
Thus (\ref{eq:matching:third:0}) requires that $a=b=0$.

Near $y=1$ we need
\begin{equation}
  \label{eq:matching:third:1}
(u_0 + \eps^{1/3} u_\third)(1-\eps^{1/2} Z) - (V_0 + \eps^{1/3} V_\third)(Z)=o(\eps^{1/3}) 
\end{equation}
for $Z$ in a range  $\eps^{-p} \ll Z \ll \eps^{-q}$ where  $0 \le p < q \le 1/2$.
We take $p=0, q=1/6$ so that $\eps^\half Z = o(\eps^\third)$.
Now by a Taylor expansion near y=1
\begin{equation} \label{eq:taylor:u:1}
  u_0(y) = - \tJ'(1) (1-y) + \Oh\prn{(1-y)^2},
\end{equation}
so by our choice of $q$
\begin{equation}
 u_0(1-\eps^{1/2} Z) = \Oh(\eps^{1/2} Z) = o(\eps^\third).
\end{equation}
Therefore,
\begin{equation}
 (u_0 + \eps^{1/3} u_\third)(1-\eps^{1/2} Z) =  \eps^{1/3} u_\third(1) +o(\eps^\third) =
 - \eps^{1/3} \lambda_\third w(1) + o(\eps^{1/3})
\end{equation}
where we have recalled that $a=b=0$ in (\ref{eq:bulk:third:linear}).
On the other hand, $V_0 \equiv 0$, so by (\ref{eq:V_third:clamped:form}), as $Z \to \infty$
\begin{equation}
 (V_0 + \eps^{1/3} V_\third)(Z) = A \eps^{1/3} (1-Z) + o(\eps^{1/3} ).
\end{equation}
Thus matching requires that $A=0$ and $\lambda_\third w(1)=0$.
Recalling from Corollary~\ref{cor:nonzero} that $w(1) \ne 0$, we obtain
$\lambda_\third =0$, and we have verified the second line of Table~\ref{tab:asymptotic:series:clamped}.


\subsection{The $\eps^\half$ Correction}

\mysubsubsec{Boundary-layer near $y=0$}
Equation \eqref{eq:reduced:equations:U}, which at order 1/2 is homogeneous, and the three homogeneous boundary conditions imply 
that $U_\half \equiv 0$.

\mysubsubsec{Boundary-layer near $y=1$}
Since $V_0 \equiv 0$, equation (\ref{eq:reduced:equations:V}) at order 1/2 is homogeneous.
By the same argument as for $V_0$ and $V_\third$, this equation and the boundary conditions yield
\begin{equation}
  \label{eq:1:1/2:general:2}
 V_\half(Z)= A(1 - Z - e^{-Z}).
\end{equation} 

\mysubsubsec{Bulk}
The $\eps^\half$ equation of \eqref{eq:reduced:equations:u}  has the
same family of solutions as in order $\eps^\third$: 
\begin{equation}
  \label{eq:bulk:1/2:general}
 u_\half(y)= -\lambda_\half w(y)+a \tJ(y)+b\tY(y)
\end{equation}
where $w$ is the solution of \eqref{eq:w-third}.

\mysubsubsec{Matching}
Matching at $y=0$ as above, we deduce that $a=b=0$ in \eqref{eq:bulk:1/2:general}.

At $y=1$ we need
\begin{equation}
  \label{eq:matching:half:1}
(u_0 + \eps^{1/3} u_\third + \eps^\half u_\half)(1-\eps^{1/2} Z) - (V_0 + \eps^{1/3} V_\third + \eps^\half V_\half)(Z)=o(\eps^{1/2})
\end{equation}
for $Z$ in a range $\eps^{-p} \ll Z \ll \eps^{-q}$ where  $0 \le p < q \le 1/2$.
We choose $p=0, q=1/4$ so that $(\eps^\half Z)^2 = o(\eps^\half)$.
Now by \eqref{eq:taylor:u:1} and (\ref{eq:bulk:1/2:general}), 
\begin{equation} \label{first-term:half}
 (u_0 + \eps^{1/3} u_\third + \eps^\half u_\half)(1-\eps^{1/2} Z) = \eps^\half \brk{-\tJ'(1)Z - \lambda_\half w(1)} +o(\eps^\half),
\end{equation}
while letting $Z \to \infty$ in (\ref{eq:1:1/2:general:2}) we deduce
\begin{equation} \label{second-term:half}
 (V_0 + \eps^{1/3} V_\third + \eps^\half V_\half)(Z) = \eps^\half A (1-Z) + o(\eps^\half).
\end{equation}
To match the expressions in \eqref{eq:matching:half:1} we need $A=\tJ'(1)$ and 
\begin{equation} \label{eq:lambda-half}
	\lambda_\half = - \tJ'(1) / w(1).
 \end{equation}
Remarkably, it follows from Lemma~\ref{lem:lambda_1/2} that $\lambda_\half = \lambda_0$.
The verification of all entries in Table~\ref{tab:asymptotic:series:clamped} is now complete.

As a check on our calculations, in Figure~\ref{fig:2_term:error:clamped} below we present a log-log plot of the error in the two-term approximation 
$\lambda_0 + \lambda_\half \eps^\half$ to the clamped eigenvalues.  
The resulting lines of slopes near 1 suggest that the next
non-vanishing correction to the eigenvalue will happen at the
\(\Oh(\eps)\) level. 
An exercise at the
end of Section~\ref{sec:pinned} includes verifying this point.

\begin{lemma}
\label{lem:lambda_1/2}
The function $w$ satisfies
\begin{equation*}
 w(1) = - \frac{\tJ'(1)}{\lambda_0}
\end{equation*}
\end{lemma}
\noindent 

\begin{proof}
Recall formula \eqref{eq:1} for $w(1)$.
Manipulating \cite[11.3.34]{AbramowitzStegun64}, we find that 
\begin{equation}
\int_0^1\tJ^2(\tau)d\tau= \int_0^1J_0^2(2\sqrt{\lambda_0\tau}) d\tau
=J_0^2( 2\sqrt{\lambda_0}) + J_1^2(2\sqrt{\lambda_0}).
\label{eq:int:j_0_2}
\end{equation}
Since $\lambda_0$ is a root of  $J_0( 2\sqrt{ \lambda_0})$, we may drop the first term.
Regarding the second we invoke \cite[9.1.28]{AbramowitzStegun64} to obtain
\begin{equation} 
 J_1^2(2\sqrt{\lambda_0})=J_0'^2(2\sqrt{\lambda_0})=\frac{1}{\lambda_0}\tJ'^2(1).
\label{eq:int:j_0_2:2}
\end{equation}
Thus we obtain $w(1)=\tY(1)\tJ'(1)\frac{\tJ'(1)}{\lambda_0}$.
Since $\tJ(1)=0$, 
\begin{equation} \label{prod-JprimeY}
 \tY(1)\tJ'(1) = -\tJ(1)\tY'(1)+\tY(1)\tJ'(1) = -\cW(1)=-1,
\end{equation}
where we have taken the Wronskian from Appendix~\ref{subsec:bessel}.
\end{proof}

\subsection{Comparisons of Various Approximate Eigenfunctions}
\label{subsec:composite:approximation}
The outer and inner solutions in our asymptotic expansions of the eigenfunctions may be combined into a 
composite expansion that gives a uniformly accurate approximation in both regions.
(See \cite[Section~5.1.8]{Hinch1991}.)
This may be formed by adding the inner and outer approximations, subtracting the common part in the matching region, and expressing
the result as a function of the outer variable.
This process is trivial at order 0: 
at $y=1$ the inner approximation, and hence the common part, vanishes, while
at $y=0$, the inner approximation is nonzero, but it equals the common part, so the uniform
approximation is simply the outer solution.
It is similarly trivial at order 1/3, so let us proceed to order 1/2, where the behavior near $y=1$ is interesting.

Inner and outer solutions are given in Table~\ref{tab:asymptotic:series:clamped}, and the common part of these expansions in the
matching region near $y=1$ is given by \eqref{first-term:half} or \eqref{second-term:half}, 
\begin{equation}
 \eps^\half \tJ'(1)\brk{1-Z}.
\end{equation}
Thus, subtracting off the common part simply cancels the two polynomial terms in the inner solution,
leaving only
the exponential.
Therefore, at order 1/2, the composite approximation of the eigenfunctions is 
\begin{equation} \label{comp-sect:unif}
 \UU_\half(y) = \tJ(y)  -\eps^\half \lambda_0 \brk{ w(y) - w(1) e^{-(1-y)/\eps^\half}}.
\end{equation}
Here we have used Lemma~\ref{lem:lambda_1/2} to rewrite the coefficient of the exponential to make 
it obvious that $\UU_\half(0)=0$.
Although this boundary condition is satisfied exactly, the derivative boundary condition $u'(0)=0$
is satisfied only to leading order; specifically $\UU_\half'(0) = \Oh(\eps^\half)$.
Such loss of accuracy in taking derivatives cannot be avoided.

Fig~\ref{fig:comp-func-compare-1/2} shows the order-0 outer, the
order-1/2 composite, and the order-1/2 inner approximations together with the numerical approximation to the first eigenfunction, while
Figure~\ref{fig:comp-deriv-compare-1/2} graphs the derivatives of the
order-0 outer approximation and the order-1/2 composite approximation,
all for $\eps = 10^{-2}$.
In anthropomorphic terms, $\UU_\half \,$ ``attempts'' to correct for the fact that the zeroth-order outer 
solution has nonzero derivative at $y=1$.
Thus, $\UU_\half$ is lowered in the interior of the interval so that it may approach $y=1$ with nearly zero slope.
Incidentally, the increase in the eigenvalue $\eps^\half \lambda_\half$ is needed to drive $\UU_\half$ towards zero more rapidly in the 
interior.

\begin{figure}[th]
\subfloat[Function.]{\scalebox{.8}{\includegraphics{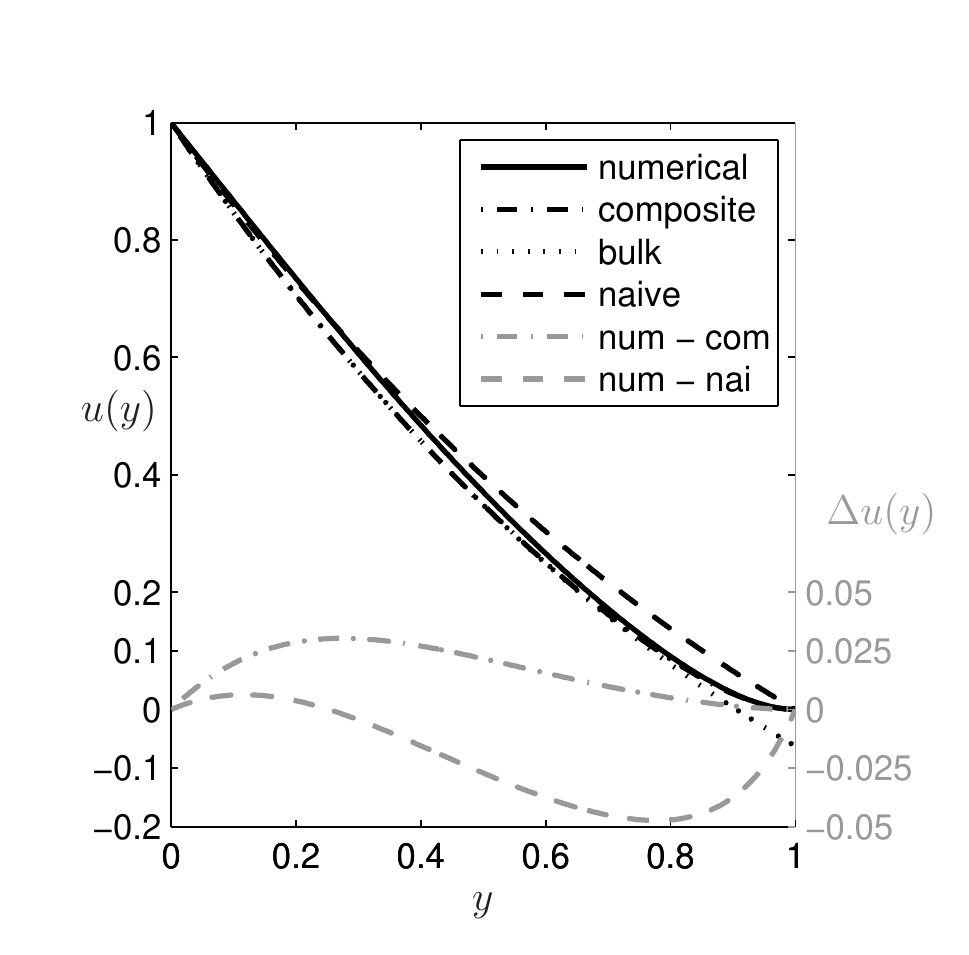}\label{fig:comp-func-compare-1/2}}}
\hfill
\subfloat[Derivative.]{\scalebox{.8}{\includegraphics{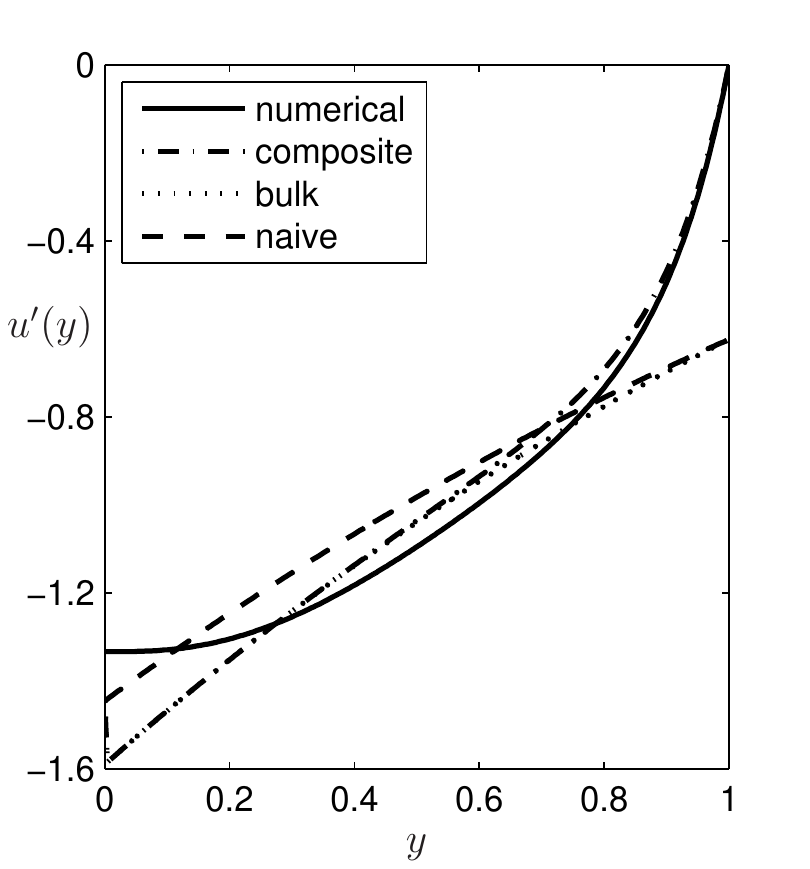}\label{fig:comp-deriv-compare-1/2}}}
\caption{Comparison of numerical, naive, bulk, and composite
  approximations of order 1/2 to the first eigenfunction, for $\eps=10^{-2}$. Both function values and first derivative are shown; in Fig~3a,  
differences are overlaid for clarification.}
\end{figure}

In Fig~\ref{fig:comp-func-compare-1/2} it may be seen that  $\UU_\half$ differs noticeably from the numerically computed eigenfunction
in the interior of the interval.
This difference is corrected by the order-2/3 terms in the series, which are driven by the boundary layer at $y=0$.
(The order-2/3 corrections for clamped boundary conditions are the same
as for the pinned case, which are calculated in Subsection~\ref{subsec:pinned:2/3}.) 
In other words, for application to the flag problem mentioned in the introduction, it is unwise to
neglect the order-2/3 corrections to the eigenfunction.

\subsection{Lessons Learned}

Having achieved our initial goal, let us pause to describe patterns in the calculations:
Suppose all terms $u_\gamma,U_\gamma,V_\gamma,\lambda_\gamma$ of order $\Oh(\eps^\gamma)$ for $\gamma < \alpha$ have been calculated, and 
consider what is needed to calculate the terms of order $\Oh(\eps^\alpha)$.

\mysubsubsec{Boundary-layer near $y=0$}
The equation (\ref{eq:reduced:equations:U}) for $U_\alpha$,
\begin{equation} \label{eq:general:U}
    U'''' -(XU')' = \text{RHS},
\end{equation}
has a four-dimensional solution space.
However, because $\int \, \Bi$ is excluded as unsuited for matching, only a three-dimensional space is available for forming  $U_\alpha$.
This function must satisfy the three boundary conditions at $X=0$ derived from (\ref{eq:boundary:conditions},\ref{eq:extraBC});
one expects these boundary conditions to determine  $U_\alpha$ uniquely.

The particular and homogeneous solutions of (\ref{eq:general:U}) play different roles.
The particular solution, which has polynomial growth as $X \to \infty$,
matches onto derivatives of lower-order terms in the bulk series.  
(For example, the particular solution $-\lambda_0 X$ in $U_{\third}$ matched onto $u_0(\eps^{\third}X)$.)
The homogeneous solution, a linear combination $A  +B \int \Ai+ C \Psi(X)$, matches onto $u_\alpha$, 
the bulk solution of the same order.  Note that as $X\to\infty$
\begin{equation}
 A +B \int_0^X \Ai(x) \, dx+ C \Psi(X) = C \log X + (A +B/3 +C \Psi_\infty) + o(1).
\end{equation}
The three-dimensional parameter space is constrained in two directions by
matching onto $u_\alpha$. 
The remaining degree of freedom provides the additional flexibility needed to satisfy all the boundary conditions.

\mysubsubsec{Boundary-layer near $y=1$}
The equation \eqref{eq:reduced:equations:V} for $V_\alpha$,
\begin{equation} \label{eq:general:V}
 V'''' - V'' = RHS,
\end{equation}
has a three-dimensional space of solutions appropriate for matching.
Solutions of the homogeneous equation are spanned by $1,Z,e^{-Z}$.
The constant function matches onto $u_\alpha$, the bulk solution of the same order; $Z$ matches onto 
$u_{\alpha-\half}(1-\eps^{\half}Z)$; and $e^{-Z}$ provides the flexibility needed to satisfy all boundary conditions without contributing 
to the matching.
If $\alpha \ge 1$, then a particular solution of \eqref{eq:general:V} will also be needed to match onto derivatives of 
$u_{\alpha-1}, u_{\alpha-3/2}, \ldots \,$.

The general solution of  \eqref{eq:general:V} has the form
\begin{equation} \label{eq:general:sol:V}
 V_\alpha(Z) = A +BZ + C e^{-Z} + V_{p}(Z).
\end{equation}
The constants may be determined by satisfying three equations: two
come from the boundary conditions at $Z=0$ and the third arises from matching terms proportional 
to $Z$ as $Z \to \infty$ with the derivative of $u_{\alpha-\half}$ at $y=1$.
Once $V_\alpha$ is determined, matching to $u_\alpha$ yields an effective boundary condition for $u_\alpha$ at $y=1$.

It is interesting to compare matching at the two end points.
In both cases, matching constrains the inner solution in two directions.
At $y=0$, both directions relate to $u_\alpha$, the outer solution at the order being calculated.
In contrast, at $y=1$, one direction relates to $u_\alpha$ and the other to $u_{\alpha-\half}$.

\mysubsubsec{Bulk}
In the equation \eqref{eq:reduced:equations:u} for $u_\alpha$, let us split off the term on the right proportional to the yet-to-be-calculated coefficient $\lambda_\alpha$:
\begin{equation*} 
    (yu')' + \lambda_0 u = -\lambda_\alpha u_0 + R.
\end{equation*}
The general solution of this equation has the form
\begin{equation} \label{eq:general:sol:u}
 u_\alpha = -\lambda_\alpha w + u_{p} +a \tJ +b\tY,
\end{equation}
where $w$ is the solution of \eqref{eq:w-third} and $a,b$ are arbitrary.
If it weren't for the complications arising from the logarithmic behavior of $\tY$ near $y=0$, the remaining steps would be extremely 
simple:
Matching $u_\alpha$ to $U_\alpha$ at $y=0$ provides two equations, which we may use to 
determine $a$ and $b$ in \eqref{eq:general:sol:u}. 
Then the effective boundary condition from matching at $y=1$
provides a linear equation for $\lambda_\alpha$, 
completing the calculation to order $\Oh(\eps^\alpha)$.
These simple ideas suffice until we encounter order $\alpha=1$; even then, the complications are rather mild.

\mysubsubsec{Matching}
To match at $y=0$ we need
\begin{equation} \label{dgs1}
 \sum_{\gamma\le \alpha} \eps^\gamma \brk{u_\gamma(\eps^\third X) - U_\gamma(X) } = o(\eps^\alpha)
\end{equation}
for an appropriate range of $X$.
Since all terms of order less than $\alpha$ have already been matched, we may focus only on terms of order exactly $\alpha$.
In the inner series, only $U_\alpha$ contributes a term of order exactly $\alpha$:
\begin{equation*}
  \sum_{\gamma\le \alpha} \eps^\gamma  U_\gamma(X)  =  \quad \ldots \quad + \eps^\alpha U_\alpha(X) + o(\eps^\alpha)
\end{equation*}
where $\ldots$ indicates terms of order less than $\alpha$, which are not relevant for calculating the $\Oh(\eps^\alpha)$-solution.
By contrast, in the outer series, in addition to $u_\alpha$, derivatives of lower-order terms in the series also 
contribute terms of exactly this order.
Of course these lower-order terms in the outer solution have already been determined; typically they are matched by the particular-solution 
part of $U_\alpha$, the arbitrary constants in the homogeneous solution playing no role.
The ``bleeding'' of lower-order terms into the order-$\alpha$ matching is also responsible for the fact that the matching interval shrinks 
as $\alpha$ increases.

Similar considerations apply to matching at $y=1$.

We have everywhere performed matching in terms of the inner variable.
It is possible to use the outer variable instead, but in our opinion the calculations are less clear:
Specifically, the terms needing to be matched at order $\alpha$ are precisely those that, 
when expressed in terms of the inner variable, are proportional 
to $\eps^\alpha$.


\section{Asymptotics for the Pinned Case}
\label{sec:pinned}
\subsection{The Low-Order Solution}
We now consider pinned
boundary conditions at $y=1$ with the same conditions at $y=0$:
\begin{equation}
  \label{eq:pinned:BC:at:1}
  u''(0)=0, \qquad u'''(0)=0, \qquad  u(1)=0,\qquad u''(1)=0,
\end{equation}
plus the normalization $u(0)=1$.
We calculate the terms in the asymptotic series through the first nontrivial correction to the eigenvalue---in this case order one.
The results are summarized in Table~\ref{tab:asymptotic:series:pinned}.

\begin{table}[h!]
\setlength{\extrarowheight}{3pt} 
\centerline{
\begin{tabular}{|C|*{4}{C}|}
\hline
\alpha& \lambda_\alpha &u_\alpha&U_\alpha&V_\alpha\\
\hline
0&\lambda_0& \tJ(y)& 1& 0\\
\third& 0&0 &-\lambda_0 X&0\\
\half&0&0&0&-\tJ'(1)Z\\
\tthirds&0&-\frac{\lambda_0^2}{6\Ai'(0)}\tJ(y)&\frac{\lambda_0^2}{4}X^2-\frac{\lambda_0^2}{2\Ai'(0)} \int_0^X
  \Ai(x)\,dx&0\\
\fsixths&0&0&0&0\\
1&\eqref{eq:pinned:lambda_1} &\eqref{the-end}&\eqref{eq:pinned:U_1}&\eqref{eq:V1:form}\\[1ex]
\hline
\end{tabular}}
\caption{The coefficients of the asymptotic expansion for the pinned
  case. The last row contains the equation numbers where the terms, too long and
  cumbersome for the table, can be found.}
\label{tab:asymptotic:series:pinned}
\end{table}

\begin{exe}
Derive the first 3 rows of Table~\ref{tab:asymptotic:series:pinned}.
\end{exe}

These first few orders are very similar to the clamped case.

\subsection{The $\eps^\tthirds$ Correction Term}
\label{subsec:pinned:2/3}
\mysubsubsec{Boundary-layer near $y=0$}
Using the values of $U_\third$ and $\lambda_\third$ from
Table~\ref{tab:asymptotic:series:pinned}, the reduced equation (\ref{eq:reduced:equations:U}) for
$U_\tthirds$ has the inhomogeneous term $-\lambda_0^2X$,
with particular solution $U_{p}=\tfrac{\lambda_0^2}{4} X^2$. 
Therefore, the relevant family of solutions is 
\begin{equation}
  U_\tthirds(X)=\tfrac{\lambda_0^2}{4} X^2 + A + B\int_0^X \Ai(\tau) d\tau  + C \cdot\Psi(X).
\end{equation}
The boundary conditions $U_\tthirds(0)= U_\tthirds''(0)= U_\tthirds'''(0)=0$ determine that $A=C=0$
and
$B=-\tfrac{\lambda_0^2}{2\Ai'(0)}$.
So we have verified the table entry for $U_\tthirds$.

\mysubsubsec{Boundary-layer near $y=1$}
Equation (\ref{eq:reduced:equations:V}) together with the two boundary conditions at $y=1$ imply that $V_\tthirds$ has the form
\begin{equation}
  \label{eq:V:reduced:2/3}
  V_\tthirds(Z)=A(1-Z-e^{-Z}).
\end{equation}

\mysubsubsec{Bulk}
Because  $\lambda_\third$ vanishes, the ODE (\ref{eq:reduced:equations:u}) for $u_\tthirds$ has only the inhomogeneous 
term $-\lambda_\tthirds u_0$.
The general solution of this equation is
\begin{equation*}
u_\tthirds(y)= -\lambda_\tthirds w(y)+a \tJ(y)+b\tY(y),
\end{equation*}
where $w(y)$ is the function that solves \eqref{eq:w-third}.

\mysubsubsec{Matching}
At $y=0$ we need equation \eqref{dgs1}, with $\alpha=2/3$, to hold
for $X$ in the range $\eps^{-p} \ll X \ll \eps^{-q}$ where $0 \le p < q \le 1/3$.
We take $p=0, q=1/9$ so that $(\eps^\third X)^3 = o(\eps^\tthirds)$, and
we focus only on terms of order exactly 2/3, using ellipsis to represent terms of lower order.
From (c) and using the fact that $w(0)=0$,
\begin{equation} \label{eq:dgs2}
 \sum_{\gamma\le 2/3} \eps^\gamma u_\gamma(\eps^\third X) = \quad \ldots \quad
    + \eps^\tthirds \brk{\frac{u_0''(0)}{2} X^2 +a +b \log(\eps^\third X) } +o(\eps^\tthirds).
\end{equation}
On the other hand, reading $U_\tthirds$ from Table~\ref{tab:asymptotic:series:pinned} and using \eqref{eq:Airy:Integrals} from the Appendix~\ref{subsec:airy},
\begin{equation} \label{eq:dgs3}
 \sum_{\gamma\le 2/3} \eps^\gamma U_\gamma(X) = \quad \ldots \quad
    + \eps^\tthirds \brk{\frac{\lambda_0^2}{4} X^2 - \frac{\lambda_0^2}{6\Ai'(0)}  } +o(\eps).
\end{equation}
We see from the Taylor series~\eqref{eq:taylor:J_0} that the quadratic terms in~\eqref{eq:dgs2} and~\eqref{eq:dgs3} agree, 
as expected.
Matching the new terms at order $2/3$ gives
\begin{equation} 
\label{eq:dgs4}
   a = -\frac{\lambda_0^2}{6\Ai'(0)}, \qquad b=0.
 \end{equation}

At $y=1$ we match 
for $Z$ in the range  $\eps^{-p} \ll Z \ll \eps^{-q}$ where  
we take $p=0, q=1/6$ so that $(\eps^\half Z)^2 = o(\eps^\tthirds)$.
Again we focus only on terms of order exactly 2/3.
Now
\begin{equation*}
  \sum_{\gamma\le 2/3} \eps^\gamma u_\gamma(1-\eps^\half Z) = \quad \ldots \quad + \eps^\tthirds u_\tthirds(1) + o(\eps^\tthirds);
\end{equation*}
no derivative of the bulk solution bleeds into the $\eps^\tthirds$ term since there is no term of order 1/6.
At the same time, by (\ref{eq:V:reduced:2/3})
\begin{equation*} 
 \sum_{\gamma\le 2/3} \eps^\gamma  U_\gamma(Z)  = \quad \ldots \quad + A \eps^\tthirds (1-Z) +  o(\eps^\tthirds).
\end{equation*}
Thus matching requires $A=0$ and $ u_\tthirds(1)=0$; the latter may be
written out as
\begin{equation*}
  - \lambda_\tthirds w(1) +a \tJ(1) = 0
\end{equation*}
where $a$ is given by (\ref{eq:dgs4}).
Of course $\tJ(1) = 0$ and $w(1) \ne 0$, thus it follows that $\lambda_\tthirds=0$ and $u_\tthirds$
is as given in Table~2.

\mysubsubsec{The composite approximation}
As discussed in Subsection~\ref{subsec:composite:approximation}, the composite approximation is given by the sum of the outer and inner
approximations minus the 
matching terms.
Taking the outer and inner solutions from Table~\ref{tab:asymptotic:series:pinned} and the matching terms near $y=0$ from
\eqref{eq:dgs3}, we find
\begin{equation} \label{tthirds-composite}
 \UU_\tthirds(y) = \tJ(y) - \eps^\tthirds \frac{\lambda_0^2}{6 \Ai'(0)} \prn{\tJ(y) - 3\int_{\eps^{-\third} y}^\infty \Ai(x) dx }.
\end{equation}
Here the integral term results from a convenient cancellation: in the matching layer, as 
$X \to \infty$, the integral over
$(0,X)$ tends to 
the integral over $(0,\infty)$, so the inner solution minus the matching terms simplifies to the integral over $(X,\infty)$.

Figure~\ref{fig:comp-func-compare-tthirds} compares the numerical solution and two other approximations to the
third eigenfunction: the naive approximation and the order-2/3 composite approximation.
The graphs of all three functions resemble that of Figure~1, and to visual accuracy they coincide; therefore  
we have plotted only differences. 
Figure~\ref{fig:comp-deriv-compare-tthirds} graphs the second
derivatives of these three approximate eigenfunctions.
In all plots $\eps = 10^{-5}$; because we are examining the \emph{third} eigenfunction, we need a
smaller $\eps$ than in Subsection~\ref{subsec:composite:approximation}
in order for the asymptotics to be meaningful---see
Section~\ref{sec:closing}.
The order-2/3 composite approximation captures most of the boundary-layer behavior of the 
eigenfunction.
In particular, observing in the Taylor series \eqref{eq:taylor:J_0} that $\tJ''(0)=\frac{\lambda_0^2}{2}$, we compute that 
\begin{equation*}
 \UU_\tthirds''(0) = - \frac{\lambda_0^4}{12 \Ai'(0)} \eps^\tthirds;
\end{equation*}
thus, the second derivative
vanishes to leading order.

\begin{figure}[th]
\subfloat[Differences.]{\scalebox{.9}{\includegraphics{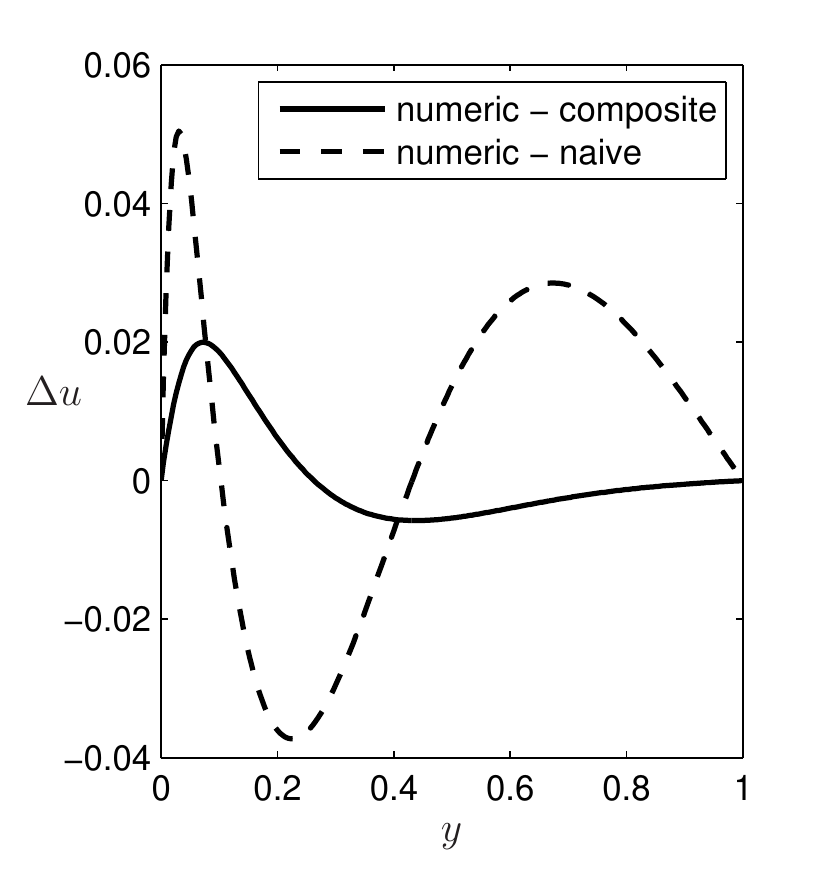}}
\label{fig:comp-func-compare-tthirds}}
\hfill
\subfloat[Derivatives.]{\scalebox{.9}{\includegraphics{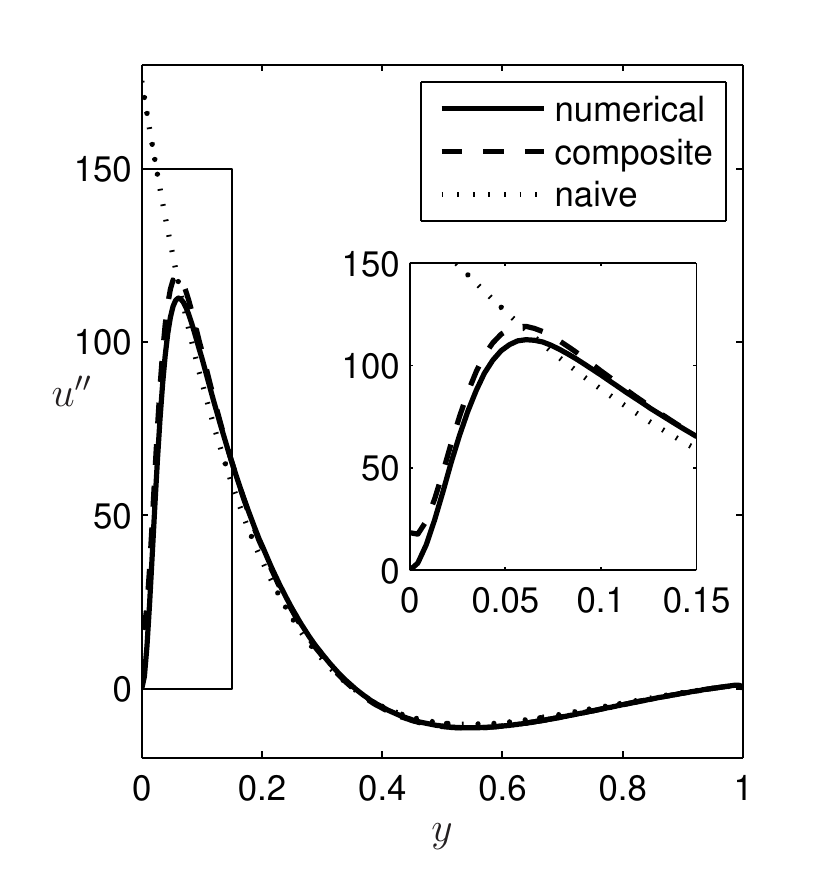}}
\label{fig:comp-deriv-compare-tthirds}}
\caption{Comparison of (a) function values and (b) second derivatives of order-2/3 approximations to the third
  eigenvalue, for $\eps = 10^{-5}$. Since the function values are too close to discern
  visually, various differences between functions
  are shown instead.}
\end{figure}

\subsection{The $\eps^\fsixths$ Correction}
\begin{exe}
\label{exe:5/6}
  Show that all the terms in this order vanish, as shown in Table~\ref{tab:asymptotic:series:pinned}.
\end{exe}

\subsection{The $\eps^1$ Correction}

This order of the calculation is pedagogically interesting because a logarithmic term enters the
series.

\mysubsubsec{The $y=0$ boundary}
Equation \eqref{eq:reduced:equations:U}
for $U_1$ is
\begin{equation}
  \label{eq:ODE:U:1}
  U_1''''-(XU_1')' =\frac{\lambda_0^3}{4}X^2-\frac{\lambda_0^3}{2\Ai'(0)} \int_0^X \Ai(x)dx.
\end{equation}
This looks quite bad due to the integral of $\Ai$ in the RHS, but in Exercise~\ref{exe:particular:U:1} you are guided through a proof that
\begin{equation}
  \label{eq:G:def}
  G(X)=-\frac{\lambda_0^3}{36}X^3+\frac{\lambda_0^3}{2\Ai'(0)} \int_0^X (X-x) \Ai(x)dx 
\end{equation}
is a particular solution to \eqref{eq:ODE:U:1}, and in Exercise~\ref{exe:G} you are asked to apply boundary conditions to conclude that
\begin{equation}
  \label{eq:pinned:U_1}
  U_1(X)=G(X) - \frac{\lambda_0^3}{2}\frac{\Ai(0)}{(\Ai'(0))^2} \int_0^X \Ai(x) dx 
+ \frac{\lambda_0^3}{3} \Psi(X).
\end{equation}

\begin{exe}
\label{exe:particular:U:1}
Find a particular solution to \eqref{eq:ODE:U:1}:
\begin{enumerate}
\item Show that the derivative of $ \int_0^X (X-x) \Ai(x)dx$ equals $\int_0^X \Ai(x)dx$.
\item Use this when applying the LHS of ODE \eqref{eq:ODE:U:1}
  to $ \int_0^X (X-x) \Ai(x)dx$ and obtain $-\int_0^X \Ai(x)dx$. 
\item Conclude that $G(X)$ in \eqref{eq:G:def} is a particular solution of \eqref{eq:ODE:U:1}
 that grows only polynomially at large values of $X$.
\end{enumerate}
\end{exe}

\begin{exe}  
\label{exe:G}
Find the $y=0$ boundary layer solution $U_1$:
 \begin{enumerate}
  \item Using Airy's differential equation, show that 
  \begin{equation}
    \label{eq:G:props}
    G(0)= 0,\quad
    G''(0)=\frac{\lambda_0^3}{2}\frac{\Ai(0)}{\Ai'(0)}, \quad 
    G'''(0)=\frac{\lambda_0^3}{3}.
  \end{equation}
\item The family of possible solutions for $U_1$ is
\begin{equation}
  \label{eq:U:1:family}
  U_1(X)=G(X)+ A +B\int_0^X \Ai(x) dx + C \Psi(X).
\end{equation}
 Using what you just found about $G$, show that \eqref{eq:pinned:U_1} satisfies the boundary conditions
$U_1(0)=U_1''(0)=U_1'''(0)=0$.
\end{enumerate}
\end{exe}

\mysubsubsec{The $y=1$ boundary}
The boundary layer equation is $V_1'''' -  V_1''=\tJ'(1)$ whose family of possible solutions is
\begin{equation*}
\label{eq:V:1:family}
  V_1=-\frac{\tJ'(1)}{2}Z^2+A+BZ+Ce^{-Z}.
\end{equation*}
The boundary conditions $V(0)=V''(0)=0$ determine $A$ and $C$
so that 
\begin{equation} 
\label{eq:V1:form}
  V_1= - \tJ'(1)\brk{\frac{1}{2}Z^2 +1 -e^{-Z}} + BZ,
\end{equation}
and by matching to the bulk solution (below), it is found that $B=0$.

\mysubsubsec{Bulk}
The bulk equation for $u_1$ is familiar, albeit with a new term on the
RHS:
\begin{equation}
(  y u_1')'+\lambda_0 u_1 = -\lambda_1 u_0 + u_0''''.
\label{eq:reduced:u:1}
\end{equation}
Its general solution is
\begin{equation}
  \label{eq:2}
u_1(y)= -\lambda_1 w(y)+v(y)+a \tJ(y)+b\tY(y),
\end{equation}
where $w(y)$ is defined by \eqref{eq:w-third} and $v(y)$ is the solution to 
\begin{equation}
(yv')'+\lambda_0 v =  \tJ''''(y), \qquad v(0)=0.
\label{eq:ODE:v}
\end{equation}
(Existence and uniqueness of $v(y)$ provided by Lemma~\ref{lem:w}.)

\mysubsubsec{Matching}
Near $y=0$ we need \eqref{dgs1} to hold 
for a range of $X$ such that $\eps^{-p}\ll X\ll\eps^{-q}$ where $0\le
  p<q\le\frac13$.
We take $p=0$, $q=1/12$ so that $(\eps^\third X)^4=o(\eps^1)$.
Using an ellipsis for terms of order less than 1, we have for $u$ 
\begin{equation}
  \label{eq:Matching:u_1:near:0}
  \sum_{\gamma\le1}\eps^\gamma u_\gamma(\eps^\third
    X)=\ldots+ \eps\brk{\frac{\tJ'''(0)}{6}X^3
  -\frac{\lambda_0^2\tJ'(0)}{6\Ai'(0)}X+ a +  b \prn{\frac{1}{3} \log \eps + \log  X } }+o(\eps^1);
\end{equation}
the cubic and linear terms come from derivatives of $u_0$ and $u_\tthirds$ respectively.
At large $X$, the behavior of $U_1$ is
  \begin{equation}
    \label{eq:U_1:large_X}
  \sum_{\gamma\le1}\eps^\gamma U_\gamma(X)=\ldots +\eps \lambda_0^3 \brk{-\frac{1}{36}X^3+\frac{1}{6\Ai'(0)}    X 
+ C_\infty + \frac{1}{3} \log X} +o(\eps^1) 
  \end{equation}
where 
\begin{equation}
\label{eq:C_tilde}
 C_\infty=  \frac12 - \frac{\Ai(0)}{6(\Ai'(0))^2} +\frac13\Psi_\infty \approx 0.06855 .
\end{equation}
Here the 1/2 in $C_\infty$ comes from a term in $G(X)$, using the relation
\begin{equation*}
 \int_0^\infty x \Ai(x) dx = - \Ai'(0)
\end{equation*}
obtained from the Airy equation.
It is readily seen from \eqref{eq:taylor:J_0} that the cubic and linear terms of \eqref{eq:Matching:u_1:near:0} and \eqref{eq:U_1:large_X} match, and the 
logarithmic terms will match if $b=\lambda_0^3 /3$.
However, it is not possible to match \eqref{eq:Matching:u_1:near:0} and \eqref{eq:U_1:large_X} with a coefficient $a$ that is independent of $\eps$.
To fix this problem we propose to augment the $u$-series with a logarithmic term: i.e., to replace $\eps u_1(y)$ by
\begin{equation} \label{u1-log}
 \eps \brk{(\log \eps) \; \hat{u}_1(y) + u_1(y)}.
\end{equation}
In Exercise~\ref{exe:no:hats} below we ask the reader to show that:
\begin{itemize}
 \item No such log terms are possible in the asymptotic series for either boundary layer or for $\lambda$.
 \item The equation for $\hat{u}_1$ is just the homogeneous version of \eqref{eq:reduced:equations:u} so that for some 
coefficients $\hat{a}, \hat{b}$
\begin{equation} \label{u-log}
 \hat{u}_1(y) = \hat{a} \tJ(y) + \hat{b} \tY(y).
\end{equation}
\end{itemize}
Therefore in matching, the RHS of \eqref{eq:Matching:u_1:near:0} should be replaced by
\begin{equation}
  \label{eq:Matching:u_log:near:0}
  \ldots+ \eps\brk{\frac{\tJ'''(0)}{6}X^3
  -\frac{\lambda_0^2\tJ'(0)}{6\Ai'(0)}X+ (a+\hat{a} \log \eps) +  (b+\hat{b} \log \eps) \prn{\frac{1}{3} \log \eps + \log  X } }+o(\eps^1).
\end{equation}
Matching between \eqref{eq:U_1:large_X} and \eqref{eq:Matching:u_log:near:0} is now possible if and only if
\begin{equation}
 a=C_\infty, \qquad b=\lambda_0^3/3, \qquad \hat{a}=-b/3= -\lambda_0^3/9, \qquad \hat{b}=0.
\end{equation}

Near the boundary at $y=1$ we match for $Z$ in the range   $1 \ll Z\ll\eps^{-1/6}$ so that both $(\eps^\half Z)^3=o(\eps^1) \;$ 
(to allow neglect of $u_0'''$) and
$\eps^\tthirds \cdot\eps^\half Z =o(\eps^1) \;$ (to allow neglect of $u_\tthirds'$).
As $Z\goto\infty$, the $u-$series has the asymptotic behavior
\begin{equation}
 \ldots \quad + \eps\brk{u_0''(1) \frac{Z^2}{2} +u_1(1)}  +o(\eps);
\end{equation}
mercifully the contribution of $\hat{u}_1$, which is proportional to $\tJ(1)$, vanishes.
From \eqref{eq:V1:form} the $V$-series has asymptotic behavior
\begin{equation}
 \ldots \quad + \eps \brk{ -\tJ'(1) \prn{ \frac{Z^2}{2} +1} +BZ} +o(\eps).
\end{equation}
According to Exercise~\ref{exe:bessel:derivs} below, the quadratic terms match; 
for the linear terms to match we need $B=0$; and for the constant terms to match we need
\begin{equation}
  u_1(1) = -\tJ'(1).
\end{equation}
Thus manipulating \eqref{eq:2}, we deduce that
\begin{equation}
  \label{eq:pinned:lambda_1}
  \lambda_1=\frac{v(1)+\frac{\lambda_0^3}{3}\tY(1)+\tJ'(1)}{w(1)}=\lambda_0\frac{\int_0^1\tJ''''(\tau)
    \tJ(\tau) d\tau}{[\tJ'(1)]^2}+ \frac13 \frac{\lambda_0^4}{[\tJ'(1)]^2} -\lambda_0,
\end{equation}
the latter equality following from the proof of Lemma~\ref{lem:w}, from Lemma~\ref{lem:lambda_1/2},
and \eqref{prod-JprimeY}.
We may regard the three terms in \eqref{eq:pinned:lambda_1} as contributions
from the fourth derivative in 
\eqref{eq:singularly_perturbed}, the boundary conditions at $y=0$, and the boundary conditions at $y=1$, respectively.
(Numerically, we find that $\lambda_1\approx 4.4280, 1887.2$ and $44403$ for the first three eigenvalues.)
Finally, we have 
\begin{equation} \label{the-end}
 (\log\eps) \hat{u}_1 (y) +u_1(y)= -\lambda_1 w(y)+v(y)+(C_\infty -\frac{\lambda_0^3}{9} \log \eps) \, \tJ(y)
  + \frac{\lambda_0^3}{3} \, \tY(y)
\end{equation}
where $C_\infty$ is given by \eqref{eq:C_tilde}.

As a check on our calculations, in Figure~\ref{fig:2_term:error:pinned} we
present a log-log plot of the error in the two-term approximation
$\lambda_0 + \lambda_1 \eps$ to the eigenvalue. The remaining error in
the approximation seems to have a slope of \(4/3\).
\begin{figure}[th]
\subfloat[Pinned case.]{\scalebox{.9}{\includegraphics{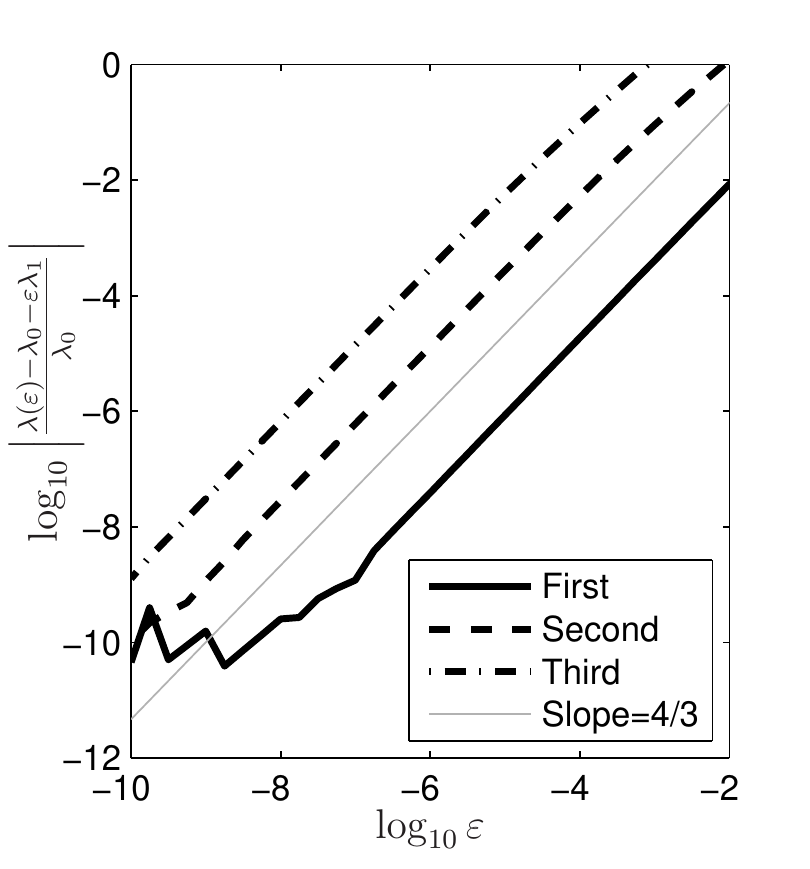}}
  \label{fig:2_term:error:pinned}}
\hfill
\subfloat[Clamped case.]{\scalebox{.9}{\includegraphics{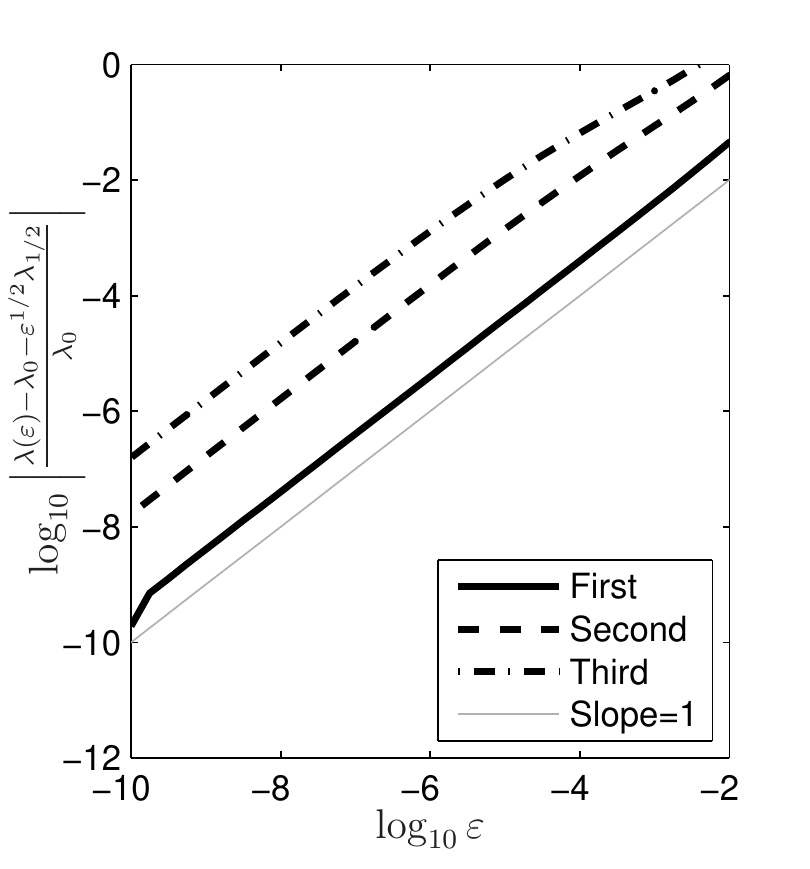}}
 \label{fig:2_term:error:clamped}}
\caption{A log-log plot of the error of the two-term approximation
      to the eigenvalues, together with a line of appropriate slope
      for visual reference.}
\end{figure}
\begin{exe} 
\label{exe:no:hats} 
Suppose that all of the series---for $u$, $U$, $V$, and $\lambda$---contain terms of order $\eps \log\eps$, with the relevant 
functions identified by a ``hat.''
\begin{enumerate}
\item Deduce that $\hat{U}_1(X) \equiv 0$ because it satisfies a homogeneous differential equation of the type 
\eqref{eq:reduced:equations:U} with homogeneous boundary conditions.
\item Deduce from the (homogeneous) equation for $\hat{V}_1$, of the type 
\eqref{eq:reduced:equations:V}, and the (homogeneous) boundary conditions that $\hat{V}_1(Z) = BZ$ for some 
constant $B$.  Show by matching with the lower-order parts of the inner solution that $B=0$.
\item Observe that $\hat{u}_1(y)$ will satisfy an ODE of the type  
\eqref{eq:reduced:equations:u}
\begin{equation*}
 (y \hat{u}_1'(y))' + \lambda_0 \hat{u}_1(y) = \hat\lambda_1 u_0.
\end{equation*}
Deduce from this equation plus matching that $\hat{\lambda}_1 = 0$ and $\hat{u}_1$ has the form \eqref{u-log}.
\end{enumerate}
\end{exe}

\begin{exe}
\label{exe:bessel:derivs} 
Using the fact that $\tJ(1)=0$, deduce from ODE \eqref{eq:reduced:equations:u} for $u_0$ that
\begin{equation*}
 \tJ'(1) = -\tJ''(1).
\end{equation*}
\end{exe}

We invite the ambitious reader to calculate the additional terms in asymptotic series for the
\emph{clamped} case through $\Oh(\eps)$:
The order-2/3 terms are identical and 
the order-5/6 terms have one minor difference, while there are some significant differences at
order-1.  
For example, the correction for the eigenvalue is given by
\begin{equation}
\label{eq:clamped:lambda_1}
  \lambda_1=\frac{v(1) +\lambda_0^2 r(1)+\frac{\lambda_0^3}{3}\tY(1) -\frac34\tJ'(1) +\lambda_0
w'(1)}{w(1)},
\end{equation}
where $r$ is defined by
\begin{equation*}
 (y r')' +\lambda_0 r =  w, \qquad r(0)=0.
\end{equation*}
Numerically, $\lambda_1\approx 6.4581, 1900$ and $44435$ for the first three eigenvalues.

\section{Closing Remarks}
\label{sec:closing}
As mentioned in the introduction, the aim of this paper was to understand the effect on the
eigenfunctions for \eqref{eq:wave:flag} of a small, but nonzero, bending resistance.
We have found that, for small epsilon, the order-2/3 composite approximation
\eqref{tthirds-composite} provides a reasonable correction to
the naive eigenfunctions.
However, the phrase ``small epsilon'' needs clarification.

Our notation so far hides the fact that problem \eqref{eq:singularly_perturbed} has an
infinite sequence of eigenvalues $\lambda^{(n)}(\eps)$ that tend to 
infinity as $n \to \infty$.
For any fixed $n$, $ \; \lambda^{(n)}(\eps)$ has an asymptotic series in $\eps$ with leading order $\lambda^{(n)}_0$ 
characterized by \eqref{eq:zeroBessel}. 
We ask, for a given $n$, how small must $\eps$ be to get into the asymptotic range.
A minimal requirement is that the order-2/3 correction to the eigenfunction be small compared to
unity: i.e., 
\begin{equation} \label{depend1}
\eps^\tthirds \lambda_0^2 \ll 1.
\end{equation}
To make this more quantitative, we invoke  the asymptotic approximation \eqref{J-bessel-asymp}
of the Bessel function to deduce that, as $n \to \infty$
\begin{equation} \label{n-depce-evalue}
    \lambda^{(n)}_0 \sim (n-\tfrac14)^2 \, \frac{\pi^2}{4}.
\end{equation}
Thus, \eqref{depend1} requires that 
\begin{equation} \label{first-estimate-eps-vs-n}
 \eps \ll n^{-6}.
\end{equation}
In other words, the meaning of ``small epsilon'' depends heavily on $n$.

In fact, for some purposes, even \eqref{first-estimate-eps-vs-n} is not 
sufficiently restrictive.
Let us define the asymptotic region by requiring that, in Figure~\ref{fig:lambda:bvp}, say in the clamped case, 
the graph of $\lambda^{(n)}(\eps) - \lambda^{(n)}_0$ 
 has converged to a line of slope 1/2; in symbols,
\begin{equation} \label{eq:7_minus_3}
 |\lambda^{(n)}(\eps) - \lambda^{(n)}_0 -\eps^\half \lambda^{(n)}_\half|  \ll \eps^\half \lambda^{(n)}_\half.
\end{equation}
We may estimate the LHS of \eqref{eq:7_minus_3} by the next term in the asymptotic series,
$\eps \lambda^{(n)}_1$, and it may be shown by estimating the terms of \eqref{eq:clamped:lambda_1}
that 
$\lambda^{(n)}_1 = \Oh(n^7)$.
Since $\lambda^{(n)}_\half = \lambda^{(n)}_0 $ and by \eqref{n-depce-evalue} the latter is
$\Oh(n^2)$, formula \eqref{eq:7_minus_3}  is equivalent to $\eps n^7 \ll \eps^\half n^2$, or
$ \eps  \ll n^{-10}$.

The restrictions in applying small-$\eps$ asymptotics are so severe that often
problem \eqref{eq:singularly_perturbed} is more accurately approximated by treating $\eps$ as a
\emph{large} parameter: i.e., regarding  \eqref{eq:singularly_perturbed} as
a perturbation of a fourth-order equation by a second-order operator.
This point of view may be developed systematically to explain the straight-line behavior in the
range 
$10^{-2} < \eps <1 $ in Figure~\ref{fig:lambda:bvp} and predict that this behavior will continue
indefinitely
as $\eps$ increases.

\section{Acknowledgements}
We are grateful to Avshalom Manela for conversations introducing us to the flag problem, and to
Manuel Kindelan for the suggestion that \eqref{eq:reduced:equations:U4} was better solved as a
boundary value problem than with shooting methods.
DGS is grateful to Universidad Carlos III de Madrid and Banco de Santander for generous sabbatical support during the academic year 2009-10.

\bibliographystyle{amsalpha}
\bibliography{general}

\providecommand{\bysame}{\leavevmode\hbox to3em{\hrulefill}\thinspace}
\providecommand{\MR}{\relax\ifhmode\unskip\space\fi MR }
\providecommand{\MRhref}[2]{%
  \href{http://www.ams.org/mathscinet-getitem?mr=#1}{#2}
}
\providecommand{\href}[2]{#2}
\begin{thebibliography}{Dow87}

\bibitem[AM05]{AM2005}
M.~Argentina and L.~Mahadevan, \emph{Fluid-flow induced flutter of a flag},
  Proceedings of the National Academy of Sciences (USA) \textbf{102} (2005),
  1829--1834.

\bibitem[AS64]{AbramowitzStegun64}
M.~Abramowitz and I.~A. Stegun (eds.), \emph{Handbook of mathematical
  functions}, 55, U.S. National Bureau of Standards, Applied Math., 1964.

\bibitem[Dow87]{Dowling1987}
A.~P. Dowling, \emph{The dynamics of towed flexible cylinders}, Journal of
  Fluid Mechanics \textbf{187} (1987), 507--532.

\bibitem[Hin91]{Hinch1991}
E.~J. Hinch, \emph{Perturbation methods}, Cambridge Texts in Applied
  Mathematics, Cambridge University Press, 1991.

\bibitem[MH09]{MH2009}
A.~Manela and M.S. Howe, \emph{On the stability and sound of an unforced flag},
  Journal of Sound and Vibration \textbf{321} (2009), no.~3--5, 994--1006.

\end{thebibliography}
%
\appendix

\section{Properties of Special Functions}
\label{sec:app:properties}
The properties of  the Bessel and Airy functions in this Appendix are
taken from Chapters~9 and~10 of \cite{AbramowitzStegun64}, respectively.

\subsection{The Airy Functions, $\Ai$ and $\Bi$}
\label{subsec:airy}
The Airy functions, $\Ai$ and $\Bi$, are  the standard solutions to the Airy equation   $u''-xu=0$.
The Wronskian of these solutions is constant:
\begin{equation}
 \Ai(x)\Bi'(x)-\Ai'(x)\Bi(x)=\frac1\pi.
\end{equation}
The Airy functions have exponential behavior for large positive $x$:
\begin{equation}
 \Ai(x) \sim \frac{e^{-\frac23x^{3/2}}}{2\sqrt\pi\,x^{1/4}}, \qquad
 \Bi(x) \sim \frac{e^{\frac23x^{3/2}}}{\sqrt\pi\,x^{1/4}}.
 \label{eq:Airy:asymptotic}
\end{equation}
The definite integral of $\Ai$ converges and is known explicitly \cite[10.4.82]{AbramowitzStegun64}:
\begin{equation}
  \label{eq:Airy:Integrals}
  \int_0^\infty \Ai(x)=\frac13.
\end{equation}
Incidentally,
\begin{equation}
 \Ai'(0)=-\frac{1}{3^\third \Gamma(\third)}.
\end{equation}

\subsection{Bessel Functions of Order Zero}
\label{subsec:bessel}

The two standard solutions to the zeroth-order Bessel Equation $u''+\frac{1}{x}u'+u=0$
are denoted $J_0$ and $Y_0$.  
Despite the singularity of Bessel's equation at $x=0$, $\; J_0(x)$ is differentiable near zero and in fact its power series 
converges for all $x$.
By contrast, as $x \to 0$,
\begin{equation*}
  Y_0(x)=\frac{2}{\pi}\prn{\log(\tfrac12 x) +\gamma} + \Oh(x^2 \log x)
\end{equation*}
where $\gamma = 0.577215665\ldots$ is Euler's constant \cite[9.1.13]{AbramowitzStegun64}.
For large $x$, these functions have the asymptotic approximation
\begin{align}
\label{J-bessel-asymp}
J_0(x) &\sim \sqrt{\frac{2}{\pi x}} \cos(x - \pi/4), \\
\label{Y-bessel-asymp}
Y_0(x) &\sim \sqrt{\frac{2}{\pi x}} \sin(x - \pi/4) 
\end{align}
In particular, $J_0$ has an infinite sequence of positive real zeros, beginning with
\begin{equation} \label{bessel-zeros}
2.40482, \;  5.52007, \; 8.65372, \;\, \dots.
 \end{equation}

In the paper we need linear combinations of Bessel functions with a slightly different
argument as defined in \eqref{eq:Bessel:short:J} and \eqref{eq:Bessel:short:Y}.
The Wronskian of $\tJ$ and $\tY$ may be computed from the fact that the Wronskian of $J_0$ and $Y_0$ equals $1/ \pi x$ and making 
the substitution $x=2\sqrt{\lambda_0 y}$: 
\begin{equation}
 \cW=\tJ(y)\tY'(y) - \tY(y) \tJ'(y) = \frac1y.
\end{equation}
$\tJ$ has the power series representation 
\begin{equation} \label{eq:taylor:J_0}
 \tJ(y) = 1 - \lambda_0 y +\prn{\frac{\lambda_0 y}{2!}}^2 - \prn{\frac{\lambda_0 y}{3!}}^3 + \ldots.
\end{equation}
This may be obtained by substitution into the power series for $J_0(x)$ or by computing coefficients recursively from the ODE \eqref{eq:reduced:expanded}.

\section{Proof: The Boundary-Layer Function $\Psi$}
\label{app:proof:Psi}

\begin{proof}[Proof of Lemma~\ref{lem:psi}:]{\itshape (Existence) } 
Equation \eqref{eq:reduced:equations:U4} has an irregular singular point at $X=\infty$ but is regular everywhere else. 
Near the singular point, this equation admits a formal series solution 
\begin{equation} \label{eq:series}
 \Psi(X)\sim	 \log X + \sum_{k=1}^{\infty} c_k X^{-3k}
\end{equation}
where $c_1=2/3$ and subsequent coefficients may be determined recursively.
Therefore there is a solution of \eqref{eq:reduced:equations:U4}, an entire function of $X$, that has the asymptotic series 
\eqref{eq:series} as $X \to \infty$, and in particular, the asymptotic behavior in the lemma.
By subtracting off a multiples of $U^{(1)},U^{(2)}$ we can also
satisfy the boundary conditions at $X=0$ without affecting the
logarithmic behavior as $X\goto\infty$.

{\itshape (Uniqueness) }
Now suppose $\tilde{\Psi}$ is another solution of \eqref{eq:reduced:equations:U4} that also satisfies 
the conditions of the lemma, and let $\Phi = \Psi - \tilde{\Psi}$.
Then $\Phi$ solves \eqref{eq:reduced:equations:U1} and $\Phi(X) = \Oh(1)$ as $X\goto\infty$. 
$\Phi$ must be a linear combination of $\{U^{(i)}, \; i=1,2,3,4\}$, where these functions are defined in Subsection~\ref{ssec:BL:Equations}.
However, $U^{(3)}$ and $U^{(4)}$ both grow more rapidly than $\Oh(1)$, and so
does any nontrivial linear combination of them.
We conclude that $\Phi = c_1 U^{(1)} + c_2 U^{(2)}$ for some constants $c_1,c_2$, and since
$\Phi(0) = \Phi''(0)=0$, we find that $c_1 = c_2 =0$.
\end{proof}

Figure~\ref{fig:Psi} shows a numerically computed graph of $\Psi(X)$.
It was found by applying a boundary-value solver (Matlab's \texttt{bvp5c}) to \eqref{eq:reduced:equations:U4}, specifying two boundary conditions at $X=0$ and
an estimate for the first derivative at $X=20$ obtained from the first 6 terms in \eqref{eq:series}.

 \begin{figure}[th]
  \scalebox{.9}{\includegraphics{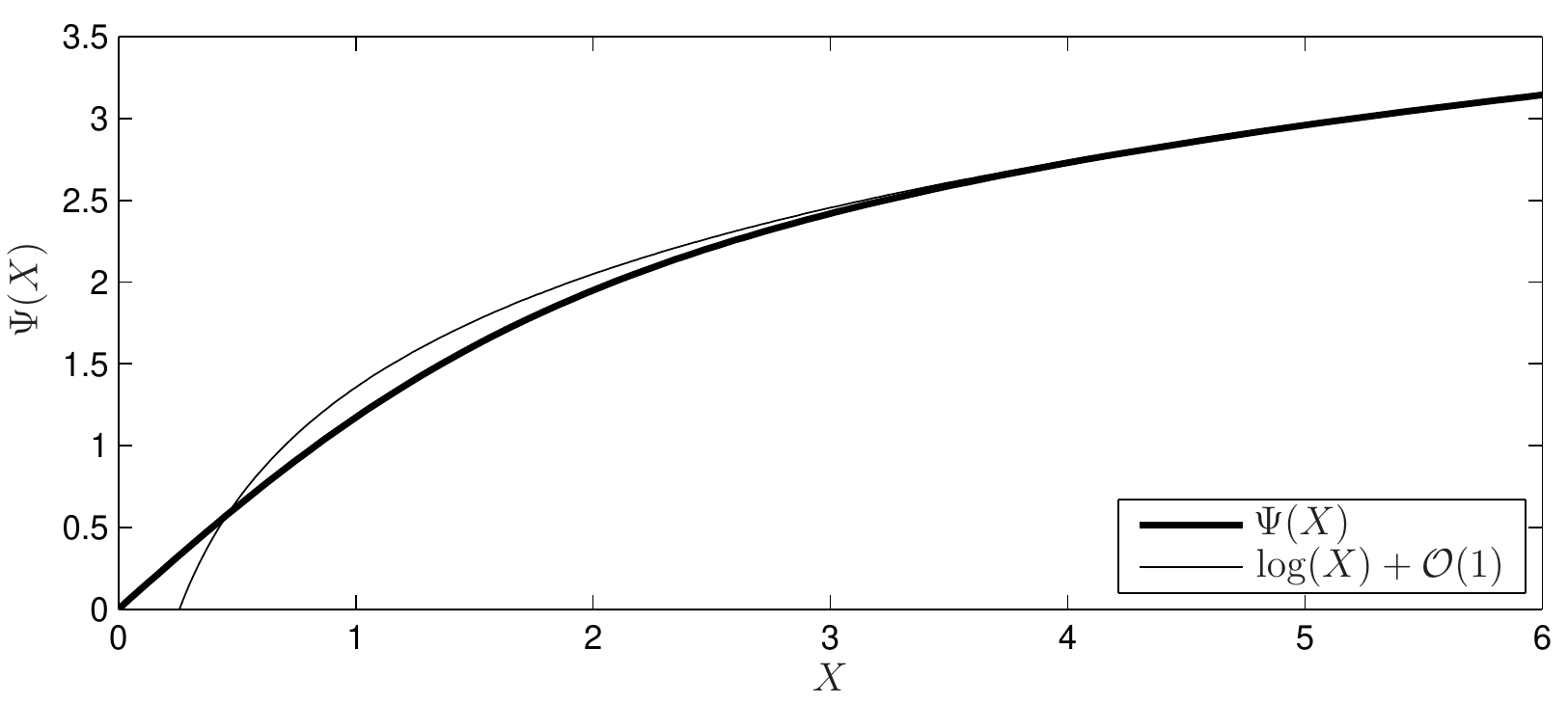}}
 \caption{A plot of $\Psi(X)$. $\Psi(X)$ is
   the solution to (\ref{eq:reduced:equations:U4}) that has
   $\Psi(0)=\Psi''(0)=0$ and a $\log(X)+\Oh(1)$ behavior for large
   $X$. Numerically, we find that $\Psi(X)-\log(X)\approx 1.3556$ at
   large $X$, so we also plot $\log(X)+1.3556$ to illustrate the convergence of $\Psi(X)$ onto it.}
\label{fig:Psi} 
\end{figure}
\end{document}